\documentclass[12pt]{article}
\usepackage{makeidx}  
\usepackage{amssymb}
\usepackage{latexsym}
\usepackage{amsmath}
\usepackage{amsfonts}
\usepackage{amsthm}
\usepackage{graphicx}
\topskip  \makeindex

\newtheorem{theorem}{Theorem}
\newtheorem{proposition}{Proposition}
\newtheorem{corollary}{Corollary}

\newtheorem{remark}{Remark}
\newtheorem{claim}{Claim}
\newtheorem{lemma}{Lemma}

\newtheorem{example}{Example}

\makeatother

\DeclareRobustCommand{\abinom}{\genfrac{\langle}{\rangle}{0pt}{}}

\begin{document}

\centerline{\Large \bf Capelli-Deruyts bitableaux}
\bigskip

\centerline{\Large \bf  and }

\bigskip

\centerline{\Large \bf the classical Capelli generators }

\bigskip

\centerline{\Large \bf of the center of the enveloping algebra $\mathbf{U}(gl(n))$}

\bigskip

\centerline{A. Brini and A. Teolis}
\centerline{\it $^\flat$ Dipartimento di Matematica, Universit\`{a} di
Bologna }
 \centerline{\it Piazza di Porta S. Donato, 5. 40126 Bologna. Italy.}
\centerline{\footnotesize e-mail: andrea.brini@unibo.it}
\medskip

\begin{abstract}

In this paper, we consider a special class of Capelli bitableaux, namely the Capell bitableaux of the form 
$\mathbf{K}^\lambda = [Der^*_\lambda|Der_\lambda] \in \mathbf{U}(gl(n))$. The main results we prove are the hook coefficient lemma and the expansion theorem.  Capelli-Deruyts bitableaux $\mathbf{K}^p_n$ of rectangular shape are of particular interest since they are central elements in the enveloping algebra $\mathbf{U}(gl(n))$. The expansion theorem implies that the central element $\mathbf{K}^p_n$ is explicitely described as a polynomial in the classical Capelli central elements $\mathbf{H}^{(j)}_n$. The hook coefficient lemma implies that the Capelli-Deruyts bitableaux 
$\mathbf{K}^p_n$ are (canonically) expressed as the products of column determinants.

\end{abstract}


\textbf{AMSC}: 17B10, 05E10, 17B35

\tableofcontents

\section{Introduction}

The study of  the center
$\boldsymbol{\zeta}(n)$ of the enveloping algebra $\mathbf{U}(gl(n))$ of the
general linear Lie algebra $gl(n, \mathbb{C})$, and  the study of the algebra
$\Lambda^*(n)$ of shifted symmetric polynomials have noble and rather independent origins and motivations.
The theme of central elements  in $\mathbf{U}(gl(n))$ is a standard one in the general theory of Lie algebras, see e.g. \cite{DIX-BR}.
It is an old and actual one, since it
is an  offspring of the celebrated Capelli identity (see e.g. \cite{Cap1-BR}, \cite{Cap4-BR}, \cite{Howe-BR}, \cite{HU-BR},
\cite{Procesi-BR}, \cite{Umeda-BR}, \cite{Weyl-BR}),
relates to its modern generalizations and applications (see e.g. \cite{ABP-BR}, \cite{KostantSahi1-BR},
\cite{KostantSahi2-BR}, \cite{MolevNazarov-BR}, \cite{Nazarov-BR}, \cite{Okounkov-BR}, \cite{Okounkov1-BR},
 \cite{UmedaCent-BR})
as well as to the theory of {\it{Yangians}} (see, e.g.  \cite{Molev1-BR}, \cite{Molev-BR}).

\emph{Capelli bitableaux} $[S | T ]$ and their variants (such as \emph{Young-Capelli bitableaux}
and \emph{double Young-Capelli bitableaux}) have been proved to be relevant in 
the study of the enveloping algebra $\mathbf{U}(gl(n)) = \mathbf{U}(gl(n), \mathbb{C})$ of the general linear Lie algebra 
and of its center $\zeta(n)$.

To be more specific,
the \emph{superalgebraic method of virtual variables} 
(see, e.g. \cite{Brini1-BR}, \cite{Brini2-BR}, 
\cite{Brini3-BR}, \cite{Brini4-BR}, \cite{Brini5-BR}, 
\cite{BriniTeolisKosz-BR}, \cite{BriniTeolisKosz-CapelliDeruyts})
allowed us to express remarkable classes 
of elements in $\mathbf{U}(gl(n))$, namely,

\begin{itemize}

\item [--] the class of \emph{Capelli bitableaux} $[S|T] \in \mathbf{U}(gl(n))$

\item [--] the class of \emph{Young-Capelli bitableaux} $[S| \fbox{$T$}] \in {\mathbf{U}}(gl(n))$

\item [--] the class of \emph{double Young-Capelli bitableaux} $[\ \fbox{$S \ | \ T$}\ ] \in {\mathbf{U}}(gl(n))$
\end{itemize}
as the images - with respect to the $Ad_{gl(n)}$-adjoint equivariant  Capelli 
\emph{devirtualization epimorphism} - of simple expressions in an enveloping 
superalgebra $\mathbf{U}(gl(m_0|m_1+n))$ (see, e.g \cite{BriniTeolisKosz-CapelliDeruyts}).

Capelli (determinantal) bitableaux are generalizations of  the famous  
\emph{column determinant element} in ${\mathbf{U}}(gl(n))$ 
introduced by Capelli in  $1887$ \cite{Cap1-BR} (see, e.g. \cite{BriniTeolisKosz-BR}).
Young-Capelli bitableaux were introduced by the present authors 
several years ago \cite{Brini2-BR}, \cite{Brini3-BR}, \cite{Brini4-BR} and might be regarded
as generalizations of the Capelli column determinant elements in ${\mathbf{U}}(gl(n))$ 
as well as of the \emph{Young symmetrizers}
of the classical representation theory of symmetric groups  (see, e.g. \cite{Weyl-BR}).
Double Young-Capelli bitableaux  play a crucial role in the study of the 
center $\boldsymbol{\zeta}(n)$ of the enveloping algebra (\cite{Brini5-BR}, \cite{BriniTeolisKosz-CapelliDeruyts}).

In plain words, the Young-Capelli bitableau $[S| \fbox{$T$}]$ is obtained 
by adding a \emph{column symmetrization} to  the Capelli bitableau $[S|T]$ 
and turn out to be a linear combination
of Capelli bitableaux (see, e.g \cite{BriniTeolisKosz-CapelliDeruyts}, Proposition $2.13$). 
The double Young-Capelli bitableau
$[\ \fbox{$S \ | \ T$}\ ]$ is obtained by adding a further \emph{row skew-symmetrization}
to the Young-Capelli bitableau $[S| \fbox{$T$}]$ 
(\cite{BriniTeolisKosz-CapelliDeruyts}, Proposition $5.1$), 
 turn out to be a linear combination
of Young-Capelli bitableaux (see, e.g \cite{BriniTeolisKosz-CapelliDeruyts}, Proposition $2.14$) and, therefore, it is in turn a linear combination of Capelli bitableaux.

Capelli bitableaux are the preimages
- with respect to the \emph{Koszul linear} ${\mathbf{U}}(gl(n))$-\emph{equivariant isomorphism} $\mathcal{K}$
from the enveloping algebra ${\mathbf{U}}(gl(n))$ to the polynomial algebra  ${\mathbb C}[M_{n,n}] \cong \mathbf{Sym}(gl(n))$ (\cite{Koszul-BR}, \cite{Brini4-BR}, \cite{BriniTeolisKosz-BR}) - of the classical
\emph{determinant bitableaux} (see, e.g. \cite{drs-BR}, \cite{DKR-BR}, \cite{DEP-BR}, \cite{rota-BR}, 
\cite{Brini1-BR}). 
Hence, they are ruled by  the \emph{straightening laws} and 
 the set of standard Capelli bitableaux is a basis of ${\mathbf{U}}(gl(n))$.

The set of standard Young-Capelli bitableaux is another relavant basis of ${\mathbf{U}}(gl(n))$ whose elements
act in a nondegenerate orhogonal way on the set of standard right symmetrized bitableaux 
(the \emph{Gordan-Capelli basis}  of ${\mathbb C}[M_{n,n}]$) and this fact leads to  explicit
complete decompositions of the semisimple ${\mathbf{U}}(gl(n))$-module
${\mathbb C}[M_{n,n}]$ (see, e.g. \cite{Brini1-BR}, \cite{Brini2-BR}).

The linear combinations of double Young-Capelli bitableaux
\begin{equation}\label{pres Schur}
\mathbf{S}_\lambda(n) =
\frac {1} {H(\tilde{\lambda})} \  \ \sum_S \ [\ \fbox{$S \ | \ S$}\ ]  \in {\mathbf{U}}(gl(n)),
\end{equation}
 where the sum is extended to all  row (strictly) increasing tableaux 
$S$ of shape $sh(S) = \widetilde{\lambda} \vdash h$, 
$\widetilde{\lambda}$ the conjugate shape/partition of $\lambda$ 
(\footnote{Given a partition (shape) 
$\lambda = (\lambda_1 \geq \lambda_2 \geq \cdots \geq \lambda_p) \vdash n$,
let $\widetilde{\lambda} = 
(\widetilde{\lambda}_1, \widetilde{\lambda}_2 \geq \cdots \geq \widetilde{\lambda}_q) \vdash n$
denote its \emph{conjugate} partition, where 
$\widetilde{\lambda}_s = \# \{t; \lambda_t \geq s \}$. }), 
are 
\emph{central elements}  of ${\mathbf{U}}(gl(n))$.

We called the elements $\mathbf{S}_{\lambda}(n)$ the \emph{Schur elements}. 
The Schur elements $\mathbf{S}_\lambda(n)$ are the preimages - with respect to the Harish-Chandra isomorphism -
of the elements of the basis of  shifted Schur polynomials  $s_{\lambda|n}^*$ 
of the algebra $\Lambda^*(n)$ of shifted symmetric polynomials \cite{Sahi1-BR},  
\cite{OkOlsh-BR}.
Hence, the Schur elements are 
the same \cite{BriniTeolisKosz-CapelliDeruyts} as the  \textit{quantum immanants}
(\cite{Sahi1-BR}, \cite{Okounkov-BR}, \cite{Okounkov1-BR}, 
\cite{OkOlsh-BR})
, first presented by  Okounkov as traces of
\textit{fusion matrices} (\cite{Okounkov-BR}, \cite{Okounkov1-BR}) and, recently, described by the present authors as linear  combinations (with explicit coefficients) of ``diagonal'' \emph{Capelli immanants} \cite{Brini5-BR}.
Presentation (\ref{pres Schur}) of Schur elements/quantum immanants doesn't involve the irreducible characters of symmetric groups.
Furthermore, it is better suited to the study of the eigenvalues on irreducible $gl(n)-$modules
and of the duality in the algebra $\boldsymbol{\zeta}(n)$,
as well as to the study of the limit $n \rightarrow \infty$,
via the \emph{Olshanski decomposition} (see,
Olshanski \cite{Olsh1-BR}, \cite{Olsh3-BR} and
Molev \cite{Molev1-BR}, pp. 928 ff.)

In this paper, we consider a special class of Capelli bitableaux,
namely the class of \emph{Capelli-Deruyts bitableaux}. These elements are Capelli bitableaux
of the form
$$
\mathbf{K}^\lambda = [ Der^*_\lambda | Der_\lambda ] \in \mathbf{U}(gl(n)),
$$
where
$\lambda = (\lambda_1 \geq \lambda_2 \geq \cdots \geq \lambda_p)$ is  a  partition with  
$\lambda_1 \leq n$,
and

\begin{itemize}

\item [--]
$Der_\lambda$ is
the \emph{Deruyts tableaux} of shape $\lambda$, that is the Young tableau of shape $\lambda$:
$$
Der_\lambda = 
\left[
\begin{array}{l}
1 \ 2 \ \ldots \ \ldots \ \ldots \ \lambda_1
\\
1 \ 2 \ \ldots \ \ldots \  \ \lambda_2
\\
 \ldots \ \ldots \ \ldots 
\\
1 \ 2 \ \ldots \ \lambda_p
\end{array}
\right]
$$

\item [--]
$ Der^*_\lambda$ is the \emph{reverse Deruyts tableaux} of shape $\lambda$, that is the Young tableau
of shape $\lambda$:
$$
Der^*_\lambda = 
\left[
\begin{array}{l}
\lambda_1 \  \ldots \ \ldots \ \ldots \ 2 \ 1
\\
\lambda_2 \  \ldots \ \ldots \ 2 \ 1
\\
 \ldots \ \ldots \ \ldots 
\\
\lambda_p \  \ldots \ 2 \ 1.
\end{array}
\right].
$$
\end{itemize}

Capelli-Deruyts bitableaux arise, in a natural way, as generalizations to 
arbitrary shapes $\lambda =  (\lambda_1 \geq \lambda_2 \geq \cdots \geq  \lambda_p)$ of the well-known 
\emph{Capelli column determinant}\footnote{The symbol $\mathbf{cdet}$
denotes the column determinat of a matrix $A = [a_{ij}]$ with noncommutative entries:
$\mathbf{cdet} (A) = \sum_{\sigma} \ (-1)^{|\sigma|}  \ a_{\sigma(1), 1}a_{\sigma(2), 2} \cdots a_{\sigma(n), n}.$} 
elements:
\begin{equation}\label{ Capelli determinant}
\mathbf{H}_n^{(n)} =
 \ \textbf{cdet}
\left(
 \begin{array}{cccc}
 e_{1, 1}+(n-1) & e_{1, 2} & \ldots  & e_{1, n} \\
 e_{2, 1} & e_{2, 2}+(n-2) & \ldots  & e_{2, n}\\
 \vdots  &    \vdots                            & \vdots &  \\
 e_{n, 1} & e_{n, 2} & \ldots & e_{n, n}\\
 \end{array}
 \right) \in \mathbf{U}(gl(n)),
 \end{equation}
introduced by Alfredo Capelli \cite{Cap1-BR} in the celebrated
identities that bear his name (see, e.g. \cite{Cap1-BR}, \cite{Cap4-BR}, \cite{Howe-BR}, \cite{HU-BR},
\cite{Procesi-BR}, \cite{Umeda-BR}, \cite{Weyl-BR},
\cite{ABP-BR}, \cite{KostantSahi1-BR},
\cite{KostantSahi2-BR}, \cite{MolevNazarov-BR}, \cite{Nazarov-BR}, \cite{Okounkov-BR}, \cite{Okounkov1-BR},
 \cite{UmedaCent-BR}).

The main results we prove are the following:

\begin{itemize}

\item [--] \textbf{The hook coefficient lemma}: let $v_{\mu}$ be a $gl(n, \mathbb{C})$-highest weight vector of weight $\mu = (\mu_1 \geq \mu_2 \geq \ldots \geq \mu_n),$ 
with $\mu_i \in \mathbb{N}$
for every $i = 1, 2, \ldots, n.$
Then, $v_{\mu}$ is an \emph{eigenvector} of the action of 
the Capelli-Deruyts bitableau $\mathbf{K^\lambda}$ with
\emph{eigenvalue} the (signed) product of \emph{hook numbers} in the Ferrers diagram of the  
partition $\mu$ (Proposition \ref{hook eigenvalue general}).

\item [--] \textbf{The expansion theorem}: the Capelli-Deruyts bitableau $\mathbf{K^\lambda} \in \mathbf{U}(gl(n))$ expands as a polynomial,
with explicit coefficients, in the \emph{Capelli generators}
$$
\mathbf{H}^{(j)}_k = \
\sum_{1 \leq i_1 < \cdots < i_j \leq k} \ \textbf{cdet}\left(
 \begin{array}{cccc}
 e_{{i_1},{i_1}}+(j-1) & e_{{i_1},{i_2}} & \ldots  & e_{{i_1},{i_j}} \\
 e_{{i_2},{i_1}} & e_{{i_2},{i_2}}+(j-2) & \ldots  & e_{{i_2},{i_j}}\\
 \vdots  &    \vdots                            & \vdots &  \\
e_{{i_k},{i_1}} & e_{{i_j},{i_2}} & \ldots & e_{{i_j},{i_j}}\\
 \end{array}
 \right)
$$
of the centers of the enveloping algebras
$\mathbf{U}(gl(k))$, $k = 1, 2, \ldots, n$, $j = 1, 2, \ldots, k$
(Theorem \ref{expansion theorem}).

\end{itemize}

Capelli-Deruyts bitableaux $\mathbf{K_n^p}$
of \emph{rectangular shape} $\lambda = n^p = (\stackrel{p \ times}{n,n,n, \cdots, n} )$
are of particular interest since they are \emph{central elements} in the
enveloping algebra $\mathbf{U}(gl(n))$. 

\begin{itemize}

\item [--]
The expansion theorem implies that the
Capelli-Deruyts bitableau $\mathbf{K_n^p}$ (with $p$ rows) equals the product of the
Capelli-Deruyts bitableau $\mathbf{K_n^{p-1}}$ (with $p-1$ rows) and the central element
$$
\mathbf{C}_n(p-1) =  \sum_{j = 0}^n \ (-1)^{n - j} (p-1)_{n - j} \ \mathbf{H}_n^{(j)}
$$
(see Corollary \ref{Teorema di sfilamento-BR}). Hence, by iterating this procedure, the central element
$\mathbf{K_n^p}$ is explicitely described as a polynomial in the classical Capelli central elements
$\mathbf{H}_n^{(j)}$ (see Corollary \ref{exp one}).

\item [--]
The hook coefficient lemma implies -via the HarishChandra isomorphism- that the element
$\mathbf{C}_n(p)$ also equals the column determinant element
$$
\mathbf{H}_n(p) = \textbf{cdet} \left[  e_{h, k} + \delta_{hk}(- p  + n - h)     
\right]_{h, k =1, \ldots, n} \in \mathbf{U}(gl(n)).
$$

Notice that
$$
\mathbf{H}_n(0) = 
\ \textbf{cdet}\left(
 \begin{array}{cccc}
 e_{1, 1}+(n-1) & e_{1, 2} & \ldots  & e_{1, n} \\
 e_{2, 1} & e_{2, 2}+(n-2) & \ldots  & e_{2, n}\\
 \vdots  &    \vdots                            & \vdots &  \\
 e_{n, 1} & e_{n, 2} & \ldots & e_{n, n}\\
 \end{array}\right) =\mathbf{H}_n^{(n)},
$$
the classical Capelli column determinant element.

From these facts, the Capelli-Deruyts bitableaux $\mathbf{K_n^p}$ are
(canonically) expressed as the  products of column determinants:
$$
\mathbf{K_n^p} =   (-1)^{n \binom {p} {2}} \ \mathbf{H}_n(p-1) \  \cdots \  \mathbf{H}_n(1) \ \mathbf{H}_n(0)
$$
(see Corollary \ref{exp two}).

\end{itemize}

The method of \emph{superalgebraic virtual variables} (\cite{Brini1-BR},
\cite{Brini2-BR}, \cite{Brini3-BR}, \cite{Brini4-BR}, \cite{Brini5-BR},
\cite{BriniTeolisKosz-BR}, \cite{BriniTeolisKosz-CapelliDeruyts})  plays a crucial role
in the present paper;  we provide a short presentation of the method in the
Appendix.

\section{The classical Capelli identities}\label{classical Capelli}

The \textit{algebra of algebraic forms $\mathbf{f}(\underline{x}_1, \ldots, \underline{x}_n)$ in $n$ 
vector variables $\underline{x}_i = (\underline{x}_{i1}, \ldots, \underline{x}_{id})$
of dimension $d$} is the polynomial algebra in $n \times d$ (commutative) variables:
$$
{\mathbb C}[M_{n,d}] =    {\mathbb C}[x_{ij}]_{i=1,\ldots,n; j=1,\ldots,d},
$$
and $M_{n,d}$ denotes the  matrix
with $n$ rows and $d$ columns with ``generic" entries $x_{ij}$:
\begin{equation}\label{matrix}
M_{n,d} = \left[ x_{ij} \right]_{i=1,\ldots,n; j=1,\ldots,d}=
 \left[
 \begin{array}{ccc}
 x_{11} & \ldots & x_{1d} \\
 x_{21} & \ldots & x_{2d} \\
 \vdots  &        & \vdots \\
 x_{n1} & \ldots & x_{nd} \\
 \end{array}
 \right].
\end{equation}

The algebra ${\mathbb C}[M_{n,d}]$ is a $\mathbf{U}(gl(n))-$module,
with respect to the action:
$$
e_{x_j,x_i} \cdot \mathbf{f} = D^{\textit{l}}_{{x_j,x_i}}(\mathbf{f}),
$$
for every $\mathbf{f} \in {\mathbb C}[M_{n,d}],$
where, for any $i, j = 1, 2, \ldots, n$, where $D^{\textit{l}}_{{x_j,x_i}}$
is the unique \emph{derivation} of the algebra ${\mathbb C}[M_{n,d}]$
such that
$$
D^{\textit{l}}_{{x_j,x_i}}(x_{hk}) = \delta_{ih} \ x_{jk},
$$
for every $k = 1, 2. \dots, d$.

\begin{proposition}\label{Capelli identity}$\mathbf{(The \ Capelli \ identities, \ 1887)}$
$$
\mathbf{H}_n^{(n)}(\mathbf{f}) = \begin{cases} 0 &\mbox{if } n > d \\
[\underline{x}_1, \ldots, \underline{x}_n] \ \Omega_n(\mathbf{f}) & \mbox{if } n = d, \end{cases}
$$
where $\mathbf{f}(\underline{x}_1, \ldots, \underline{x}_n) \in {\mathbb C}[M_{n,d}]$ 
is an \emph{algebraic form} (polynomial)
in the $n$ vector variables $\underline{x}_i = (x_{i1}, \ldots, x_{id})$
of dimension $d$, and,  if $d = n$, $[\underline{x}_1, \ldots, \underline{x}_n]$ is the \emph{bracket}
$$
[\underline{x}_1, \ldots, \underline{x}_n] = 
det 
\left[
\begin{array}{lll}
x_{11} \ \ldots \ x_{1n}
\\
\ \vdots \ \ \ \ \ \vdots  \ \ \   \vdots
\\
x_{n1} \ \ldots \ x_{nn}\\
\end{array}
\right],
$$
and $\Omega_n$ is the \emph{Cayley $\Omega$-process} 
$$
\Omega_n = det
\left[
\begin{array}{lll}
\frac {\partial} {\partial x_{11}} \ \ldots \ \frac {\partial} {\partial x_{1n}}
\\
\ \ \vdots \ \ \ \ \  \vdots  \ \ \ \ \  \vdots
\\
 \frac {\partial} {\partial x_{n1}} \ \ldots \ \frac {\partial} {\partial x_{nn}}\\
\end{array}
\right].
$$
\end{proposition}\qed

From \cite{BriniTeolisKosz-BR}, we recall that the determinant element $\mathbf{H}_n^{(n)}$ can be written 
as the (one row) \emph{Capelli-Deruyts bitableau} $[n \ldots 2 1|1 2 \ldots n] $ (\cite{Brini2-BR}, see also \cite{Brini5-BR},  \cite{Koszul-BR}).
\begin{proposition}\label{one row Capelli-Deruyts} The element
$$
\mathbf{H}_n^{(n)} =
 \ \textbf{cdet}\left(
 \begin{array}{cccc}
 e_{1, 1}+(n-1) & e_{1, 2} & \ldots  & e_{1, n} \\
 e_{2, 1} & e_{2, 2}+(n-2) & \ldots  & e_{2, n}\\
 \vdots  &    \vdots                            & \vdots &  \\
 e_{n, 1} & e_{n, 2} & \ldots & e_{n, n}\\
 \end{array}
 \right) \in \mathbf{U}(gl(n))
$$
equals the \emph{one row Capelli-Deruyts bitableau} (see, e.g. Subsection \ref{citazione 5} below)
$$
[n \ldots 2 1|1 2 \ldots n] 
= \mathfrak{p} \left( e_{n, \alpha} \cdots e_{2, \alpha} e_{1, \alpha} \cdot
e_{\alpha, 1}e_{\alpha, 2} \cdots e_{\alpha, n} \right),
$$
where
$\mathfrak{p}$ denotes the \emph{Capelli devirtualization epimorphism}
(see, e.g. Subsection \ref{citazione 4} below).
\end{proposition}

From eq. (\ref{ Capelli determinant}) and Proposition \ref{one row Capelli-Deruyts}, it follows:
\begin{proposition}\label{Cl Capelli} We have:
\begin{enumerate}

\item \label{hook}
Let $v_{\mu}$ be a $gl(n, \mathbb{C})$-highest weight vector of weight 
$\mu = (\mu_1 \geq \mu_2 \geq \ldots \geq \mu_n),$ 
with $\mu_i \in \mathbb{N}$
for every $i = 1, 2, \ldots, n.$
Then $v_{\mu}$ is an \emph{eigenvector} of the action of $\mathbf{H}_n^{(n)}$ with
\emph{eigenvalue}:
$$
 (\mu_1  + n -1)(\mu_2 + n -2) \cdots
\mu_n.
$$
In symbols,
$$
\mathbf{H}_n^{(n)} \cdot v_{\mu} = 
 \left( (\mu_1+ n -1)(\mu_2 + n -2) \cdots
\mu_n \right)  \ v_{\mu}.
$$

\item
The element
$\mathbf{H}_n^{(n)}$ is \emph{central} in the enveloping algebra
$\mathbf{U}(gl(n))$.
\end{enumerate}
\end{proposition}

\section{The Capelli-Deruyts bitableaux in $\mathbf{U}(gl(n))$}

We generalize the  \emph{one row}
Capelli bitableau $\mathbf{H}_n^{(n)}  = [n \ldots 2 1|1 2 \ldots n]$ to arbitrary shapes (partitions)
$$
\lambda = (\lambda_1 \geq \lambda_2 \geq \cdots \geq \lambda_p), 
\qquad \lambda_i \in \mathbb{Z}^+.
$$

\subsection{Capelli-Deruyts bitableaux $\mathbf{K}^{\lambda}$
of  shape $\lambda$.}

Given a partition(shape) $\lambda = \lambda_1 \geq \lambda_2 \geq \cdots \geq \lambda_p$, 
we recall that the \emph{Deruyts tableaux} of shape $\lambda$ is the Young tableau
\begin{equation}\label{Deruyts}
Der_\lambda = ( \underline{\lambda_1}, \underline{\lambda_2}, \ldots, \underline{\lambda_p} )
\end{equation}
and the \emph{reverse Deruyts tableaux} of shape $\lambda$ is the Young tableau
$$
Der^*_\lambda = ( \underline{\lambda_1}^*, \underline{\lambda_2}^*, \ldots, \underline{\lambda_p}^* ),
$$
where
$$
\underline{\lambda_i} = 1 \ 2 \ \cdots \ \lambda_i
$$
and
$$
\underline{\lambda_i}^* = \lambda_i \ \cdots \ 2 \ 1, 
$$
for every $i = 1, 2, \ldots, p$.

The \emph{Capelli-Deruyts bitableau} $\mathbf{K}^\lambda$ is the Capelli bitableau in $\mathbf{U}(gl(n))$,
$n \geq \lambda_1$:
$$
\mathbf{K}^\lambda = [ Der^*_\lambda | Der_\lambda ] = 
\mathfrak{p} \big( e_{Der^*_\lambda C_\lambda} \cdot e_{ C_\lambda Der_\lambda} \big),
$$
where $\mathfrak{p}$ denotes the Capelli devirtualization epimorphism
and $e_{Der^*_\lambda C_\lambda}, \ e_{ C_\lambda Der_\lambda}$ are
\emph{bitableax monomials} (see., e.g.
Subsection \ref{citazione 5}, eq.
(\ref{BitMon})).

\begin{example} Let $\lambda = (3, 2, 2)$. Then
\begin{multline*}
\mathbf{K}^{(3, 2, 2)} 
=
\left[
\begin{array}{lll}
 3 \ 2 \ 1
\\
 2 \ 1
\\
 2 \ 1\\
\end{array}
\right| \left.
\begin{array}{l}
1 \ 2 \ 3
\\
1 \ 2   
\\
1 \ 2 
\end{array}
\right] =
\\
=
\mathfrak{p} \big( e_{3  \alpha_1} e_{2  \alpha_1} e_{1  \alpha_1} e_{2  \alpha_2} e_{1  \alpha_2} 
e_{2  \alpha_3} e_{1  \alpha_3}
\cdot
e_{\alpha_1 1} e_{\alpha_1 2} e_{\alpha_1 3} e_{\alpha_2 1} 
e_{\alpha_2 2} e_{\alpha_3 1} e_{\alpha_3 2}
\big)
\in \mathbf{U}(gl(n)), \quad n \geq 3,
\end{multline*}
where $\alpha_1, \alpha_2, \alpha_3$ are (arbitrary, distinct)
\emph{positive virtual symbols}.
\end{example}\qed

\begin{remark}\label{Deruyts remark} Given a Young tableau
\begin{equation}\label{Deruyts general}
T =
\left[
\begin{array}{llllll}
x_{11} \ x_{12} \ \cdots \ \cdots \ \cdots \ x_{1 \lambda_1}
\\
x_{21} \ x_{22} \ \cdots \ \cdots \ \cdot \cdot \ x_{2 \lambda_2}
\\
\vdots
\\
x_{i1} \ x_{i2} \ \cdots \ \cdots \ \cdot  \ x_{i \lambda_i}
\\
\vdots
\\
x_{p1} \ x_{p2} \ \cdots \ \cdots  \ x_{p \lambda_p}
\end{array} 
\right], \ x_{ij} \in X,
\end{equation}
of shape $\lambda = (\lambda_1 \geq \lambda_2 \geq \cdots \geq \lambda_p)$
over the set $X$ is said to be of \emph{Deruyts type} whenever
$$
\{x_{i1}, \ x_{i2}, \ldots , \ x_{i \lambda_i} \}
\subseteq\{x_{i-1 \ 1}, \ x_{i-1 \ 2}, \ldots , \ x_{i-1 \ \lambda_{i-1}} \},
$$
for $i = 2, \ldots, p$.

Clearly, any tableau of Deruyts  type (\ref{Deruyts general}) can be regarded as a Deruyts tableau
(\ref{Deruyts}), by suitably renaming and reordering the entries.
\end{remark}

\subsection{The Capelli-Deruyts bitableaux $\mathbf{K_n^p}$
of rectangular shape $\lambda = n^p$}
Given any positive integer $p$, we define the {\it{rectangular Capelli/Deruyts bitableau}}, with $p$ rows
of length $\lambda_1 = \lambda_2 = \cdots = \lambda_p = n$:
$$
\mathbf{K_n^p} =
\left[
\begin{array}{l}
n \ n-1 \ \ldots \ 3 \ 2 \ 1\\
n \ n-1 \ \ldots \ 3 \ 2 \ 1\\
\cdots\\
\\
\cdots\\
n \ n-1 \ \ldots \ 3 \ 2 \ 1\\
\end{array}\\
\right| \left.
\begin{array}{l}
1 \ 2 \ 3 \ \dots \ n-1 \ n \\
1 \ 2 \ 3 \ \dots \ n-1 \ n  \\
\cdots\\
\\
\cdots\\
1 \ 2 \ 3 \ \dots \ n-1 \ n  \\
\end{array}
\right] \in \mathbf{U}(gl(n)).
$$

 From Proposition \ref{rappresentazione aggiunta-BR}, we  infer:

\begin{proposition}

The elements $\mathbf{K_n^p}$ are central in $\mathbf{U}(gl(n))$.
\end{proposition}

Set, by definition, $\mathbf{K_n^0} = \mathbf{1}.$

\section{The hook eigenvalue Theorem for Capelli-Deruyts bitableaux}

Any rectangular Capelli-Deruyts bitableau $\mathbf{K_n^p}$ well behaves on 
$gl(n, \mathbb{C})$-highest weight vectors (compare with Proposition \ref{Cl Capelli}, item 1)).

\begin{theorem}\label{The hook coefficient lemma-BR}({\bf{The hook coefficient lemma}})

Let $v_{\mu}$ be a highest weight vector of weight $\mu = (\mu_1 \geq \mu_2 \geq \ldots \geq \mu_n),$ 
with $\mu_i \in \mathbb{N}$
for every $i = 1, 2, \ldots, n.$
Then $v_{\mu}$ is an \emph{eigenvector} of the action of $\mathbf{K_n^p}$ with
\emph{eigenvalue} the (signed) product of \emph{hook numbers} in the Ferrers diagram of the  
partition $\mu$:
$$
(-1)^{\binom{p} {2} n} \ \left( \prod_{j = 1}^{p} \ (\mu_1 - j + n )(\mu_2 - j + n -1) \cdots
(\mu_n - j + 1) \right).
$$
In symbols,
$$
\mathbf{K_n^p} \cdot v_{\mu} = 
(-1)^{\binom{p} {2} n} \ \left( \prod_{j = 1}^{p} \ (\mu_1 - j + n )(\mu_2 - j + n -1) \cdots
(\mu_n - j + 1) \right)  \ v_{\mu}.
$$

\end{theorem}

Theorem \ref{The hook coefficient lemma-BR}
generalizes to arbitrary Capelli-Deruyts bitableaux $\mathbf{K_\lambda}$ of shape $\lambda$
as follows:
\begin{proposition}\label{hook eigenvalue general}
Let $v_{\mu}$ be a highest weight vector of weight $\mu = (\mu_1 \geq \mu_2 \geq \ldots \geq \mu_n),$ with $\mu_i \in \mathbb{N}$
for every $i = 1, 2, \ldots, n.$ Let $\lambda = (\lambda_1 \geq \cdots \geq \lambda_p)$ be a partition(shape).
Then
\begin{multline*}
\mathbf{K}^\lambda \cdot v_{\mu} = \ (-1)^{\lambda_p(\lambda_{p-1} + \cdots + \lambda_1) + 
\lambda_{p-1}(\lambda_{p-2} + \cdots + \lambda_1) + \cdots + \lambda_2 \lambda_1} \ \times 
\\
\times
\left( \prod_{i = 1}^{p} \ (\mu_1 - i + \lambda_i)(\mu_2 - i + \lambda_i -1) \cdots
(\mu_{\lambda_i} - i + 1) \right)  \ v_{\mu}.
\end{multline*}
\end{proposition}

\section{The factorization Theorem for Capelli-Deruyts bitableaux}

Let $J = \{j_1 < j_2 < \dots < j_k \} \subseteq \underline{n} = \{1, 2; \ldots, n \}$. 
With a slight abuse of notation, we write $\underline{J}$ for the increasing word
$\underline{J} = j_1 j_2 \cdots j_k$ and $\underline{J}^*$ for the decreasing word
$\underline{J}^* = j_k \cdots j_2 j_1$. 

Given a partition 
$\lambda = (\lambda_1 \geq \lambda_2 \geq \cdots \geq \lambda_p)$, set
$|\lambda| = \lambda_1 + \lambda_2 + \cdots + \lambda_p$.

We have
$$
\mathbf{K}^\lambda= \left[
\begin{array}{lll}
\underline{\lambda_1}^*
\\
\underline{\lambda_2}^*
\\
\vdots
\\
\underline{\lambda_p}^*
\end{array} 
\right|\left.
\begin{array}{lll}
\underline{\lambda_1}
\\
\underline{\lambda_2}
\\
\vdots
\\
\underline{\lambda_p}
\end{array}
\right]
$$
and, consistently, we write, for $J \subseteq M$, 
$$
\left[
\begin{array}{l}
\mathbf{K}^\lambda\\
J
\end{array}
\right] = 
\left[
\begin{array}{lll}
\underline{\lambda_1}^*
\\
\underline{\lambda_2}^*
\\
\vdots
\\
\underline{\lambda_p}^*
\\
\underline{J}^*
\end{array} 
\right|\left.
\begin{array}{lll}
\underline{\lambda_1}
\\
\underline{\lambda_2}
\\
\vdots
\\
\underline{\lambda_p}
\\
\underline{J}
\end{array}
\right],  \quad [J] = [\underline{J}^* | \underline{J}].
$$

\begin{theorem}\label{MAIN} $(\bf{The \ row \ insertion \ theorem})$ Let $m \leq \lambda_p$.
Given $M \subseteq \underline{\lambda_p}$, $|M| = m$, we have
$$
\left[ M^* | M \right] \ \mathbf{K}^\lambda =
\sum_{k = 0}^m
\
\left< p \right>_{m-k}
\
\sum_{J;\ J \subseteq M;\ |J|=k}
\ (-1)^{|\lambda|k}
\left[
\begin{array}{l}
\mathbf{K}^\lambda
\\
J
\end{array}
\right],
$$
where $\left< p \right>_j$ denonotes the \emph{raising factorial}
$$
\left< p \right>_j = p(p+1) \cdots (p+j-1).
$$
\end{theorem}\qed

\begin{theorem}\label{expansion theorem}$(\bf{The \ expansion \ theorem})$
Let $m \leq \lambda_p$. 
Given $M \subseteq \underline{\lambda_p}$, $|M| = m$, we have
\begin{equation*}\label{inverse}
(-1)^{|\lambda| m} \ 
\left[
\begin{array}{l}
\mathbf{K}^\lambda
\\
M
\end{array}
\right] =
\sum_{k=0}^m \ 
\
(-1)^{m-k}
\left( p \right)_{m-k}
\
\sum_{J;\ J \subseteq M;\ |J|=k}
\
[\underline{J}^* | \underline{J}] 
\ \mathbf{K}^\lambda,
\end{equation*}
where $( p )_j$ denonotes the \emph{falling factorial}
$$
( p )_j = p(p-1) \cdots (p-j+1).
$$
\end{theorem}

\begin{proof} 

By Theorem \ref{MAIN},
\begin{multline*}
\sum_{k=0}^m \ 
\
(-1)^{m-k}
\left( p \right)_{m-k}
\
\sum_{J;\ J \subseteq \ M;\ |J|=k}
\
[J] \ \mathbf{K}^\lambda
=
\\
= 
\sum_{k=0}^m \ (-1)^{m-k}
\left( p \right)_{m-k}
\
\sum_{J;\ J \subseteq \ M;\ |J|=k}
\
\sum_{i = 0}^k
\
\left< p \right>_{k-i}
\
\sum_{I;\ I \subseteq J;\ |I|=i}
\ (-1)^{|\lambda| i}
\left[
\begin{array}{l}
\mathbf{K}^\lambda
\\
I
\end{array}
\right] =
\\
= \sum_{i=0}^m \
\sum_{k=i}^m \ \sum_{I;\ I \subseteq \ M;\ |I|=i} \
\big( \sum_{J;\ M \ \supseteq J \ \supseteq I;\ |J|=k} 
\ (-1)^{m-k}
\left( p \right)_{m-k}
\
\left< p \right>_{k-i}
\big)
\ (-1)^{|\lambda|i}
\left[
\begin{array}{l}
\mathbf{K}^\lambda
\\
I
\end{array}
\right] =
\\
= \sum_{i=0}^m \ \sum_{I;\ I \subseteq \ M;\ |I|=i} \ \big(
\sum_{k=i}^m \ (-1)^{m-k}
\left( p \right)_{m-k}
\
\left< p \right>_{k-i}
\
\binom{m-i}{k-i}
\big)
\ (-1)^{|\lambda|i}
\left[
\begin{array}{l}
\mathbf{K}^\lambda
\\
I
\end{array}
\right] =
\\
= \sum_{i=0}^m \ \sum_{I;\ I \subseteq \ M;\ |I|=i} \ \big( (m-i)!
\sum_{k=i}^m \ (-1)^{m-k}
\binom{p}{m-k}
\abinom{p}{k-i}
\big)
\ (-1)^{|\lambda|i}
\left[
\begin{array}{l}
\mathbf{K}^\lambda
\\
I
\end{array}
\right] =
\\
= \sum_{i=0}^m \ \sum_{I;\ I \subseteq \ M;\ |I|=i} \ \big( (m-i)!
\ \delta_{m-i,0}
\big)
\ (-1)^{|\lambda|i}
\left[
\begin{array}{l}
\mathbf{K}^\lambda
\\
I
\end{array}
\right] =
\\
= \sum_{i=0}^m \ \sum_{I;\ I \subseteq \ M;\ |I|=i} \ \big( (m-i)!
\ \delta_{m,i}
\big)
\ (-1)^{|\lambda|i}
\left[
\begin{array}{l}
\mathbf{K}^\lambda
\\
I
\end{array}
\right] = (-1)^{|\lambda|m}
\left[
\begin{array}{l}
\mathbf{K}^\lambda
\\
M
\end{array}
\right].
\end{multline*}
\end{proof}

\begin{example} 

\begin{enumerate}
\

\item We have
\begin{align*}
\left[ 2 1 | 1 2 \right]
\left[
\begin{array}{lll}
3 \ 2 \ 1
\\
2 \ 1
\end{array} 
\right|\left.
\begin{array}{lll}
1 \ 2 \ 3
\\
1 \ 2
\end{array}
\right] &= 
6 
\left[
\begin{array}{lll}
3 \ 2 \ 1
\\
2 \ 1
\end{array} 
\right|\left.
\begin{array}{lll}
1 \ 2 \ 3
\\
1 \ 2
\end{array}
\right] 
+
2
\left[
\begin{array}{lll}
3 \ 2 \ 1
\\
2 \ 1
\\
1
\end{array} 
\right|\left.
\begin{array}{lll}
1 \ 2 \ 3
\\
1 \ 2
\\
1
\end{array}
\right] 
\\
&+ 
2
\left[
\begin{array}{lll}
3 \ 2 \ 1
\\
2 \ 1
\\
2
\end{array} 
\right|\left.
\begin{array}{lll}
1 \ 2 \ 3
\\
1 \ 2
\\
2
\end{array}
\right] 
+ 
\left[
\begin{array}{lll}
3 \ 2 \ 1
\\
2 \ 1
\\
2 \ 1
\end{array} 
\right|\left.
\begin{array}{lll}
1 \ 2 \ 3
\\
1 \ 2
\\
1 \ 2
\end{array}
\right].
\end{align*}
\

\item We have
\begin{align*}
\left[
\begin{array}{lll}
3 \ 2 \ 1
\\
2 \ 1
\\
2 \ 1
\end{array} 
\right|\left.
\begin{array}{lll}
1 \ 2 \ 3
\\
1 \ 2
\\
1 \ 2
\end{array}
\right] &= 2 \left[
\begin{array}{lll}
3 \ 2 \ 1
\\
2 \ 1
\end{array} 
\right|\left.
\begin{array}{lll}
1 \ 2 \ 3
\\
1 \ 2
\end{array}
\right]
-
2
\left[ 1 | 1 \right]
\left[
\begin{array}{lll}
3 \ 2 \ 1
\\
2 \ 1
\end{array} 
\right|\left.
\begin{array}{lll}
1 \ 2 \ 3
\\
1 \ 2
\end{array}
\right] 
\\
&
-
2 
\left[ 2 | 2 \right]
\left[
\begin{array}{lll}
3 \ 2 \ 1
\\
2 \ 1
\end{array} 
\right|\left.
\begin{array}{lll}
1 \ 2 \ 3
\\
1 \ 2
\end{array}
\right]
+
\left[ 2 \ 1 | 1 \ 2 \right]
\left[
\begin{array}{lll}
3 \ 2 \ 1
\\
2 \ 1
\end{array} 
\right|\left.
\begin{array}{lll}
1 \ 2 \ 3
\\
1 \ 2
\end{array}
\right].
\end{align*}
\end{enumerate}

\end{example}

\section{The center $\zeta(n)$ of $\mathbf{U}(gl(n))$} 

\subsection{The Capelli generators of the center $\zeta(n)$ of $\mathbf{U}(gl(n))$}

In the enveloping algebra $\mathbf{U}(gl(n))$, given any increasing $k$-tuple
integers $1 \leq i_1  < \cdots< i_k \leq n.$ 

We recall that the 
column determinant
\begin{equation*}\label{virtual Capelli elements}
\textbf{cdet}\left(
 \begin{array}{cccc}
 e_{{i_1},{i_1}}+(k-1) & e_{{i_1},{i_2}} & \ldots  & e_{{i_1},{i_k}} \\
 e_{{i_2},{i_1}} & e_{{i_2},{i_2}}+(k-2) & \ldots  & e_{{i_2},{i_k}}\\
 \vdots  &    \vdots                            & \vdots &  \\
e_{{i_k },{i_1}} & e_{{i_k},{i_2}} & \ldots & e_{{i_k},{i_k}}\\
 \end{array}
 \right) \in \mathbf{U}(gl(n))
\end{equation*}
equals the \emph{one-row} Capelli-Deruyts bitableau 
$$
[i_k i_{k-1}  \cdots i_1|i_1 \cdots i_{k-1} i_k] =
\mathfrak{p} \left(e_{i_k \alpha} e_{i_{k-1} \alpha}  \cdots e_{i_1 \alpha}  
e_{\alpha i_1} \cdots e_{\alpha i_{k-1}} e_{\alpha i_k} \right)
\in \mathbf{U}(gl(n))
$$
(see, e.g. \cite{BriniTeolisKosz-BR}).

Consider the $k$-th \emph{Capelli element}
$$
\mathbf{H}_n^{(k)} = \
\sum_{1 \leq i_1 < \cdots < i_k \leq n} \ \textbf{cdet}\left(
 \begin{array}{cccc}
 e_{{i_1},{i_1}}+(k-1) & e_{{i_1},{i_2}} & \ldots  & e_{{i_1},{i_k}} \\
 e_{{i_2},{i_1}} & e_{{i_2},{i_2}}+(k-2) & \ldots  & e_{{i_2},{i_k}}\\
 \vdots  &    \vdots                            & \vdots &  \\
e_{{i_k},{i_1}} & e_{{i_k},{i_2}} & \ldots & e_{{i_k},{i_k}}\\
 \end{array}
 \right)
$$

Clearly, we have
\begin{equation}\label{The classical Capelli generators}
\mathbf{H}_n^{(k)} = \ \sum_{1 \leq i_1 < \cdots < i_k \leq n} \ [ i_k \cdots i_2 i_1 | i_1 i_2 \cdots i_k ]. 
\end{equation}

We recall the following fundamental result, proved by  Capelli in  two  
papers (\cite{Cap2-BR}, \cite{Cap3-BR}) with deceiving titles.

\begin{proposition} $(\mathbf{Capelli,  \ 1893})$
Let  $\zeta(n)$ denote  be \emph{center} of
$\mathbf{U}(gl(n))$.
We have:

\begin{itemize}

\item[--] The elements $\mathbf{H}_n^{(k)}$, $k = 1, 2, \ldots, n$ belong
to the center $\zeta(n)$.
\item[--] The subalgebra $\zeta(n)$ of
$\mathbf{U}(gl(n))$ is the polynomial algebra
$$
\zeta(n)  \  = \ \mathbb{C}[\mathbf{H}_n^{(1)}, \mathbf{H}_n^{(2)}, \ldots, \mathbf{H}_n^{(n)}], 
$$
where
$$
\mathbf{H}_n^{(1)}, \mathbf{H}_n^{(2)}, \ldots, \mathbf{H}_n^{(n)} 
$$
is a set of algebraically independent generators of  $\zeta(n)$.

\end{itemize}
\end{proposition}

\subsection{The factorization Theorem for rectangular Capelli-Deruyts bitableaux $\mathbf{K_n^{p}}$}

The crucial result in this section is that Capelli-Deruyts bitableaux $\mathbf{K_n^{p}}$ of \emph{rectangular}
shape $\lambda = n^p$ expand into \emph{commutative} polynomials
in the Capelli elements $\mathbf{H}_n^{(j)}$, with explicit coefficients.

The next result was announced, without proof, in \cite{BriUMI-BR}.
By eq. (\ref{The classical Capelli generators}), it is a special case of Theorem \ref{expansion theorem}.

\begin{corollary}$({\bf{Expansion Theorem}})$ \label{Teorema di sfilamento-BR}

Let $p \in \mathbb{N}$ and set $\mathbf{H}_n^{(0)} = \mathbf{1},$ by definition. 
The following identity in $\zeta(n)$
holds:
\begin{equation*} \label{equazione di sfilamento-BR}
 \mathbf{K_n^{p}} =  
(-1)^{n(p-1)} \ \mathbf{C}_n(p-1)  \
\mathbf{K_n^{p-1}},
\end{equation*}
where,
given $p \in \mathbb{N}$, 
\begin{equation} \label{relazioni lineari-BR}
\mathbf{C}_n(p-1) =  \sum_{j = 0}^n \ (-1)^{n - j} (p-1)_{n - j} \ \mathbf{H}_n^{(j)}.
\end{equation}
where
$$
(m)_k = m(m-1) \cdots (m-k+1), \ m, k \in \mathbb{N}
$$
denotes the falling factorial coefficient.
\end{corollary}

If $p = 0$, eq. (\ref{relazioni lineari-BR}) collapses to
$$
\mathbf{K_n^{1}} = \mathbf{H}_n^{(n)} =  \mathbf{C}_n(0).
$$

Notice that the linear relations (\ref{relazioni lineari-BR}), for $p = 0, \ldots, n-1$, yield a nonsingular triangular
coefficients matrix.

\begin{corollary} The subalgebra $\zeta(n)$ of
$\mathbf{U}(gl(n))$ is the polynomial algebra
$$
\zeta(n)  \  = \ \mathbb{C}[\mathbf{C}_n(0), \mathbf{C}_n(1), \ldots, \mathbf{C}_n(n - 1)], 
$$
where
$$
\mathbf{C}_n(0), \mathbf{C}_n(1), \ldots, \mathbf{C}_n(n - 1) 
$$
is a set of algebraically independent generators of  $\zeta(n)$.
\end{corollary}

\begin{corollary}\label{exp one} The rectangular Capelli-Deruyts bitableau $\mathbf{K_n^p}$ equals the \emph{commutative
polynomial} in the Capelli generators:
$$
\mathbf{K_n^p} = (-1)^{n \binom{p} {2} }   \ \mathbf{C}_n(p-1) \ \cdots \ \mathbf{C}_n(1) \ \mathbf{C}_n(0).
$$
\end{corollary}

\begin{example}Let $n=3$, $p=2$. Then
$$
\mathbf{K_3^2} = 
\left[
\begin{array}{lll}
 3 \ 2 \ 1
\\
 3 \ 2 \ 1
\end{array}
\right| \left.
\begin{array}{l}
1 \ 2 \ 3
\\
1 \ 2 \ 3 
\end{array}
\right] =  \ - \ \mathbf{C}_3(1) \ \mathbf{C}_3(0) =
\left( \mathbf{H}_3^{(2)} - \mathbf{H}_3^{(3)} \right) \mathbf{H}_3^{(3)}.
$$ \qed
\end{example}

\subsection{The Harish-Chandra isomorphism and the algebra $\Lambda^*(n)$ of shifted symmetric polynomials}

In this subsection we follow  A. Okounkov and G. Olshanski
\cite{OkOlsh-BR}.

As in the classical context of the algebra $\Lambda(n)$ of symmetric
polynomials in $n$ variables $x_1, x_2, \ldots, x_n$, the algebra
$\Lambda^*(n)$ of {\it{shifted symmetric polynomials}} is an algebra
of polynomials $p(x_1, x_2, \ldots, x_n)$  but the ordinary symmetry
is replaced by the {\it{shifted symmetry}}:
$$
 f(x_1, \ldots , x_i, x_{i+1}, \ldots, x_n) = f(x_1, \ldots , x_{i+1} - 1, x_i + 1,
 \ldots, x_n),
$$
for $i = 1, 2, \ldots, n - 1.$

The {\it{shifted elementary symmetric polynomials}} are the elements
of $\Lambda^*(n)$

\begin{itemize}

\item[--]for every $r \in \mathbb{Z}^+$,
\begin{equation*}\label{shifted elementary-BR}
\mathbf{e}_k^{*}(x_1, x_2, \ldots, x_n)
= \sum_{1 \leq i_1 < i_2 < \cdots < i_k \leq n} \ (x_{i_1}  + k  - 1)
(x_{i_2}  + k - 2) \cdots (x_{i_k}),
\end{equation*} 

\item[--] $\mathbf{e}_0^{*}(x_1, x_2, \ldots, x_n) = \mathbf{1}$.
\end{itemize}

The {\it{Harish-Chandra isomorphism}} is the algebra isomorphism
$$
\chi_n : \zeta(n) \longrightarrow \Lambda^*(n), \qquad  \ A \mapsto \chi_n(A),
$$
$\chi_n(A)$ being the shifted symmetric polynomial such that, for every highest weight module $V_{\mu}$,
the evaluation $\chi_n(A)(\mu_1, \mu_2, \ldots , \mu_n)$ equals the eigenvalue of
$A \in \zeta(n)$ in  $V_{\mu}$ (\cite{OkOlsh-BR}, Proposition $\mathbf{2.1}$).

\subsection{The Harish-Chandra isomorphism interpretation of Theorem \ref{The hook coefficient lemma-BR} and
Theorem \ref{expansion theorem}}

Notice that
$$
\chi_n(\mathbf{H}_n^{(r)}) =  \mathbf{e}_r^{*}(x_1, x_2, \ldots, x_n) \in \Lambda^*(n),
$$
for every $r =  1, 2, \ldots, n.$

Furthermore, from Theorem \ref{The hook coefficient lemma-BR}  it  follows

\begin{corollary}\label{eq quozienti-BR}
$$
\chi_n(\mathbf{K_n^p}) = (-1)^{{\binom{p} {2} n}} \left(  \prod_{j = 1}^{p} \
(x_1 - j + n )(x_2 - j + n -1) \cdots (x_n - j - 1) \right).
$$
\end{corollary}

By Corollary \ref{Teorema di sfilamento-BR}, we have
\begin{equation*}\label{product}
\chi_n(\mathbf{K_n^{p+1}}) = \ \chi_n(\mathbf{C}_n(p)) \ \chi_n(\mathbf{K_n^{p}}).
\end{equation*}
and Corollary \ref{eq quozienti-BR} implies

\begin{proposition}\label{HC di C}For every $p \in \mathbb{N}$,
\begin{equation*}\label{eq fattorizzazione-BR}
\chi_n(\mathbf{C}_n(p)) = (x_1 - p + n - 1)(x_2 - p + n - 2) \cdots (x_n - p).
\end{equation*}
\end{proposition}

\begin{proposition}
 The set
$$
\chi_n(\mathbf{C}_n(0)), \ \chi_n(\mathbf{C}_n(1)), \ \ldots \ , \ \chi_n(\mathbf{C}_n(n - 1))
$$
is a system of algebraically independent generators of the
ring $\Lambda^*(n)$ of shifted symmetric polynomials in the variables $x_1, x_2, \ldots, x_n.$
\end{proposition}

Given $p \in \mathbb{N}$, consider the column determinant
\begin{equation}\label{centrality p}
\mathbf{H}_n(p) = \textbf{cdet}\left(
 \begin{array}{cccc}
 e_{1,1} - p + (n-1) & e_{1,2} & \ldots  & e_{1,n} \\
 e_{2,1} & e_{2,2} - p + (n-2) & \ldots  & e_{2,n}\\
 \vdots  &    \vdots                            & \vdots &  \\
e_{n,1} & e_{n,2} & \ldots & e_{n,n} - p\\
 \end{array}
 \right).
\end{equation}

We recall a standard result (for an elementary proof see e.g.  \cite{Umeda-BR}):

\begin{proposition}\label{centrality Umeda-BR}

For every $p \in \mathbb{N}$, the element
$$
\mathbf{H}_n(p) = \textbf{cdet} \left[  e_{h, k} + \delta_{hk}(- p  + n - h)     
\right]_{h, k =1, \ldots, n} \in \mathbf{U}(gl(n)).
$$
is  central.
In symbols, $\mathbf{H}_n(p)  \in \zeta(n).$

\end{proposition}

Equation (\ref{centrality p}), Proposition \ref{centrality Umeda-BR} and Proposition \ref{HC di C}
imply
$$
\chi_n(\mathbf{H}_n(p)) = (x_1 - p + n - 1)(x_2 - p + n - 2) \cdots (x_n - p) = \chi_n(\mathbf{C}_n(p)).
$$
Hence, we get the well-known identity (see, e.g. \cite{Molev1-BR}):

\begin{corollary}

For every $p \in \mathbb{N}$, we have
\begin{align*}
\mathbf{H}_n(p) & = \textbf{cdet} \left[  e_{h, k} + \delta_{hk}(- p  + n - h)     \right]_{h, k =1, \ldots, n}
\\
& = \sum_{j = 0}^n \ (-1)^{n - j} (p)_{n - j} \ \mathbf{H}_n^{(j)} =  \mathbf{C}_n(p).
\end{align*}
\end{corollary}

\begin{corollary} The subalgebra $\zeta(n)$ of
$\mathbf{U}(gl(n))$ is the polynomial algebra
$$
\zeta(n)  \  = \ \mathbb{C}[\mathbf{H}_n(0), \mathbf{H}_n(1), \ldots, \mathbf{H}_n(n-1)], 
$$
where
$$
\mathbf{H}_n(0), \mathbf{H}_n(1), \ldots, \mathbf{H}_n(n-1) 
$$
is a set of algebraically independent generators of  $\zeta(n)$.
\end{corollary}

\begin{corollary}\label{exp two}  The rectangular Capelli-Deruyts bitableau $\mathbf{K_n^p}$ equals 
the \emph{product 
of column determinants}:
$$
\mathbf{K_n^p} =   (-1)^{n \binom{p} {2}} \ \mathbf{H}_n(p-1) \  \cdots \  \mathbf{H}_n(1) \ \mathbf{H}_n(0).
$$\end{corollary}

\begin{example}Let $n=3$, $p=2$. Then 
\begin{multline*}
\mathbf{K_3^2} = 
\left[
\begin{array}{lll}
 3 \ 2 \ 1
\\
 3 \ 2 \ 1
\end{array}
\right| \left.
\begin{array}{l}
1 \ 2 \ 3
\\
1 \ 2 \ 3 
\end{array}
\right]
=   - \ \mathbf{H}_3(1) \ \mathbf{H}_3(0)  =
\\
=
-
\textbf{cdet}\left(
 \begin{array}{cccc}
 e_{1,1}  + 1 & e_{1,2} &  e_{1,3} \\
 e_{2,1} & e_{2,2}   &  e_{2,3}\\
 e_{3,1} & e_{3,2} &  e_{3,3} - 1\\
 \end{array}
 \right)
\ \textbf{cdet}\left(
 \begin{array}{cccc}
 e_{1,1}  + 2 & e_{1,2} &  e_{1,3} \\
 e_{2,1} & e_{2,2}  + 1 &  e_{2,3}\\
 e_{3,1} & e_{3,2} &  e_{3,3}\\
 \end{array}
 \right).
\end{multline*} \qed
\end{example}

Corollaries \ref{exp one} and \ref{exp two} generalize to 
Capelli-Deruyts bitableaux $\mathbf{K}^\lambda$ of arbitrary shape $\lambda$.

Theorem \ref{expansion theorem} implies:

\begin{proposition}\label{equazione di sfilamento generale}Let $n \in \mathbb{Z}$, 
$\lambda = (\lambda_1 \geq \lambda_2  \geq \cdots \geq \lambda_p)$,
$\lambda_1  \leq n$. Set $\lambda' = (\lambda_1 \geq \lambda_2  \geq \cdots \geq \lambda_{p-1})$.
Then
$$
\mathbf{K}^\lambda  = (-1)^{\lambda_p(\lambda_{p-1} + \cdots + \lambda_{1})}
\ \mathbf{C}_{\lambda_p}(p-1) \ \mathbf{K}^{\lambda'},
$$
where
$$
\mathbf{C}_{\lambda_p}(p-1)  = \sum_{j=0}^{\lambda_p} \ (-1)^{\lambda_p - j} \
(p-1)_{\lambda_p - j} \ \mathbf{H}_{\lambda_p}^{(j)}.
$$
\end{proposition}\qed

\begin{corollary}\label{Expansion in the C} Let $n \in \mathbb{Z}$, $\lambda = (\lambda_1 \geq \lambda_2  \geq \cdots \geq \lambda_p)$,
$\lambda_1  \leq n$. For $i = 1, 2, \ldots, p$, set
$$
\mathbf{C}_{\lambda_i}(i-1)  = \sum_{j=0}^{\lambda_i} \ (-1)^{\lambda_i - j} \
(i-1)_{\lambda_i - j} \ \mathbf{H}_{\lambda_i}^{(j)}.
$$
Then,

\begin{enumerate}

\item The element
$\mathbf{C}_{\lambda_i}(i-1)$ is \emph{central} in the enveloping algebra $\mathbf{U}(gl(\lambda_i))$,
for $i = 1, 2, \ldots, p$.

\item The Capelli-Deruyts bitableau $\mathbf{K}^\lambda$ equals the polynomial in the 
\emph{Capelli elements} $\mathbf{H}_{\lambda_i}^{(j)}$:
$$
\mathbf{K}^\lambda  = (-1)^{\lambda_p(\lambda_{p-1} + \cdots + \lambda_{1}) + \cdots + \lambda_{2}\lambda_{1}}
\ \mathbf{C}_{\lambda_p}(p-1) \cdots
\mathbf{C}_{\lambda_2}(1) \ \mathbf{C}_{\lambda_1}(0).
$$
\end{enumerate}
\end{corollary} \qed

\begin{example}Let $n=3$, $\lambda = (3, 2)$
and let
$$
\mathbf{K}^{(3, 2)} = \ \left[
\begin{array}{lll}
 3 \ 2 \ 1
\\
 2 \ 1 
\end{array}
\right| \left.
\begin{array}{l}
1 \ 2 \ 3
\\
1 \ 2  
\end{array}
\right].
$$
Then,
$$
\mathbf{K}^{(3, 2)}  =   
\ 
\mathbf{C}_2(1) \ \mathbf{C}_3(0) =
\left( \mathbf{H}_2^{(2)} - \mathbf{H}_2^{(1)} \right) \mathbf{H}_3^{(3)}.
$$
\end{example}\qed

For  $i = 1, 2, \ldots, p$, consider 
the center $\zeta(\lambda_i)$ of $\mathbf{U}(gl(\lambda_i))$ and
the Harish-Chandra isomorphisms
$$
\chi_{\lambda_i} : \zeta(\lambda_i) \longrightarrow \Lambda^*(\lambda_i).
$$

Proposition \ref{hook eigenvalue general} and  Proposition \ref{equazione di sfilamento generale}
imply:
\begin{equation}\label{HC general}
\chi_{\lambda_i} \big( \mathbf{C}_{\lambda_i}(i-1) \big) = \ 
(x_1 - i + \lambda_i)(x_2 - i + \lambda_i -1) \cdots
(x_{\lambda_i} - i + 1).
\end{equation}

Proposition \ref{centrality Umeda-BR} implies that the element
$$
\mathbf{H}_{\lambda_i}(i-1) = \textbf{cdet} \left[  e_{h, k} + \delta_{hk}(\lambda_i - i - h + 1)     
\right]_{h, k =1, \ldots, \lambda_i} \in \mathbf{U}(gl(\lambda_i)).
$$
is  central in the enveloping algebra $\mathbf{U}(gl(\lambda_i))$.
In symbols, $\mathbf{H}_n(p)  \in \zeta(\lambda_i).$

Clearly,
$$
\chi_{\lambda_i} \big( \mathbf{H}_{\lambda_i}(i-1) \big) = \ 
(x_1 - i + \lambda_i)(x_2 - i + \lambda_i -1) \cdots
(x_{\lambda_i} - i + 1),
$$
and, therefore, from eq. (\ref{HC general}), we have
\begin{corollary}
$\mathbf{H}_{\lambda_i}(i-1) = \mathbf{C}_{\lambda_i}(i-1)$.
\end{corollary}

From Corollary \ref{Expansion in the C}, we have

\begin{corollary}\label{exp three}  The  Capelli-Deruyts bitableau $\mathbf{K_\lambda}$ equals 
the \emph{product 
of column determinants}:
$$
\mathbf{K^\lambda} =  (-1)^{\lambda_p(\lambda_{p-1} + \cdots + \lambda_{1}) + \cdots + \lambda_{2}\lambda_{1}}
\ \mathbf{H}_{\lambda_p}(p-1) \cdots
\mathbf{H}_{\lambda_2}(1) \ \mathbf{H}_{\lambda_1}(0).
$$\end{corollary}

\begin{example} We have
\begin{multline*}
\mathbf{K}^{(3, 2)}  =  
= 
\left[
\begin{array}{lll}
 3 \ 2 \ 1
\\
 2 \ 1 \ 
\end{array}
\right| \left.
\begin{array}{l}
1 \ 2 \ 3
\\
1 \ 2 \  
\end{array}
\right]
= 
 \mathbf{H}_2(1) \ \mathbf{H}_3(0) 
=
\\
=
\textbf{cdet}\left(
 \begin{array}{cccc}
 e_{1,1}   & e_{1,2}  \\
 e_{2,1} & e_{2,2} - 1  \\
\\
 \end{array}
 \right)
\ \textbf{cdet}\left(
 \begin{array}{cccc}
 e_{1,1}  + 2 & e_{1,2} &  e_{1,3} \\
 e_{2,1} & e_{2,2}  + 1 &  e_{2,3}\\
 e_{3,1} & e_{3,2} &  e_{3,3}\\
 \end{array}
 \right).
\end{multline*}
\end{example}

\subsection{Polynomial identities}

\

Let $t$ be a variable and consider the polynomial
$$
\mathbf{H}_n(t) = \textbf{cdet}\left(
 \begin{array}{cccc}
 e_{1,1} - t + (n-1) & e_{1,2} & \ldots  & e_{1,n} \\
 e_{2,1} & e_{2,2} - t + (n-2) & \ldots  & e_{2,n}\\
 \vdots  &    \vdots                            & \vdots &  \\
e_{n,1} & e_{n,2} & \ldots & e_{n,n} - t\\
 \end{array}
 \right) =
$$
$$
= \textbf{cdet} \left[  e_{i, j} + \delta_{ij}(- t  + n - i)     \right]_{i, j =1, \ldots, n}
$$
with coefficients in $\mathbf{U}(gl(n)).$

\

\ 

\begin{corollary}(see, e.g. \cite{Umeda-BR})\label{eqUMEDA-BR}
In the polynomial algebra $\zeta(n)[t]$,  the following identity holds:
$$
 \mathbf{H}_n(t)  = \sum_{j = 0}^n \ (-1)^{n - j}  \ \mathbf{H}_n^{(j)} \ (t)_{n - j},
$$
where, for every $k \in \mathbb{N}$, $(t)_k = t(t-1) \cdots (t-k+1)$ denotes the $k-$th 
falling factorial polynomial.
\end{corollary}

\begin{corollary}
In the polynomial algebra $\Lambda^*(n)[t]$,  the following identity holds:
$$
(x_1 - t + n - 1)(x_2 - t + n - 2) \cdots (x_n - t) =
\sum_{j = 0}^n \ (-1)^{n - j}  \ \mathbf{e}_j^{*}(x_1, x_2, \ldots, x_n) \ (t)_{n - j}.
$$
\end{corollary}

Following Molev \cite{Molev-BR} Chapt. {\bf 7} (see also    Howe and Umeda \cite{HU-BR}), consider the ``Capelli determinant''
$$
{\mathcal C}_n(s) = \textbf{cdet}\left(
 \begin{array}{cccc}
 e_{1,1} + s  & e_{1,2} & \ldots  & e_{1,n} \\
 e_{2,1} & e_{2,2} + s - 1 & \ldots  & e_{2,n}\\
 \vdots  &    \vdots                            & \vdots &  \\
e_{n,1} & e_{n,2} & \ldots & e_{n,n} + s - (n-1)\\
 \end{array}
 \right) =
$$
$$
= \textbf{cdet} \left[  e_{i, j} + \delta_{ij}(s  - i + 1)     \right]_{i, j =1, \ldots, n},
$$
 regarded as a  polynomial in the variable $s$.

By the formal (column) Laplace rule, the coefficients ${\mathcal C}_n^{(h)} \in \mathbf{U}(gl(n))$ in the expansion
$$
{\mathcal C}_n(s)  = s^n +  {\mathcal C}_n^{(1)} s^{n - 1} + {\mathcal C}_n^{(2)} s^{n - 2} + \ldots + {\mathcal C}_n^{(n)},
$$
are the sums of  the minors:
$$
{\mathcal C}_n^{(h)} = \sum_{1 \leq i_1 < i_2 < \ldots < i_h \leq n} \ {\mathcal M}_{ i_1, i_2, \ldots, i_h},
$$
where  ${\mathcal M}_{ i_1, i_2, \ldots, i_h}$ denotes the column determinant of the submatrix of the matrix ${\mathcal C}_n(0)$
obtained by selecting the rows and the columns with indices $i_1 < i_2 < \ldots < i_h.$

Since ${\mathcal C}_n(s) = \mathbf{H}_n(-s + (n-1))$,
from  Proposition \ref{eqUMEDA-BR}
it follows:

\begin{corollary} \label{expansion-BR}

$$
{\mathcal C}_n(s)  =   \sum_{j = 0}^n \ (-1)^{n - j} (-s + (n-1))_{n - j} \ \mathbf{H}_n^{(j)}.
$$
\end{corollary}

\begin{corollary} 
We have:

\begin{itemize}
\item [--]
The  elements ${\mathcal C}_n^{(h)}, \ h = 1, 2, \ldots, n$ are central and provide a system of algebraically
independent generators of $\zeta(n).$

\item [--]
$
\chi_n({\mathcal C}_n^{(h)}) = \bar{\mathbf{e}}_h(x_1, x_2, \ldots, x_n) = \mathbf{e}_h(x_1 , x_2  - 1, \ldots, x_n - (n - 1)),
$
where $\mathbf{e}_h$ denotes the $h-$th elementary symmetric polynomial.
\end{itemize}

\end{corollary}

\subsection{The shaped Capelli  central elements $\mathbf{K}_{\lambda}(n)$}

Given a  partition $\lambda = (\lambda_1, \lambda_2, \ldots, \lambda_p)$,
 $\lambda_1 \leq n$, consider the \emph{shaped Capelli elements} (see \cite{BriniTeolisKosz-BR})
$$
\mathbf{K}_{\lambda}(n) = \sum_S \  \mathfrak{p}\big( e_{S,C_{\lambda}^*} \cdot e_{C_{\lambda}^*,S}   \big)
= \sum_S \ [S|S] \in \mathbf{U}(gl(n)),
$$
where the sum is extended to all row-increasing tableaux $S$, $sh(S) = \lambda$.

Notice that the elements 
$\mathbf{K}_{\lambda}(n)$ are \emph{radically different}
from the elements 
$\mathbf{H}_{\lambda}(n) = \mathbf{H}_{\lambda_1}(n) \cdots \mathbf{H}_{\lambda_p}(n)$
and are \emph{radically different} from the elements $\mathbf{K}^{\lambda}$.

Since the adjoint representation acts by derivation, we have
$$
ad(e_{i j})\big(  \sum_S \   e_{S,C_{\lambda}^*} \cdot e_{C_{\lambda}^*,S}   \big) = 0,
$$
for every  $e_{i j} \in gl(n)$ and,  then,
from Proposition \ref{rappresentazione aggiunta-BR}, it follows

\begin{proposition}
The elements $\mathbf{K}_\lambda(n)$ are  central in $\mathbf{U}(gl(n))$.
\end{proposition}

Let $\boldsymbol{\zeta}(n)^{(m)}$ be the $m$-th filtration element of the center
$\boldsymbol{\zeta}(n)$ of $\mathbf{U}(gl(n))$.

Clearly,
$
\mathbf{K}_\lambda(n), \mathbf{H}_{\lambda}(n) \in \boldsymbol{\zeta}(n)^{(m)}
$
if and only if $m \geq |\lambda|.$


\begin{proposition}

\begin{equation*}\label{sfilamenti 1}
\mathbf{K}_{\lambda}(n) = \pm \mathbf{H}_{\lambda}(n) + \sum \ c_{\lambda,\mu} \mathbf{F}_{\mu}(n),
\end{equation*}
where $\mathbf{F}_{\mu}(n) \in \boldsymbol{\zeta}(n)^{(m)}$ for some $m< |\lambda|$.
\end{proposition} 
\begin{proof}
Immediate from Corollary \ref{filtration corollary}.
\end{proof}

Therefore, the central elements $\mathbf{K}_{\lambda}(n)$, $|\lambda| \leq m$ are linearly independent
in $\boldsymbol{\zeta}(n)^{(m)}$, and the next result
follows at once.

\begin{proposition}\label{K basis}
The set
$$
\big\{ \mathbf{K}_{\lambda}(n); \lambda_1 \leq n \ \big\}
$$
is a linear basis of the center $\boldsymbol{\zeta}(n).$
\end{proposition}

Let $\mathcal{K}$ be the \emph{Koszul equivariant isomorphism}  \cite{BriniTeolisKosz-BR}
$$
\mathcal{K} : \mathbf{U}(gl(n)) \rightarrow {\mathbb C}[M_{n,n}],
$$
\begin{equation}\label{Koszul}
\mathcal{K} : [S|S] \mapsto (S|S).
\end{equation}

Clearly, the Koszul map $\mathcal{K}$ induces, by restriction, an isomorphism
from the center $\boldsymbol{\zeta}(n)$ of $\mathbf{U}(gl(n))$ to the algebra
${\mathbb C}[M_{n,n}]^{ad_{gl(n)}}$ of
$ad_{gl(n)}-$invariants in
${\mathbb C}[M_{n,n}]$.

Consider to the polynomial
\begin{align*}
\mathbf{h}_k(n) =& \ \sum_{1 \leq i_1 < \cdots < i_k \leq n} \  ( i_k \cdots i_2 i_1 | i_1 i_2 \cdots i_k )
\\
=& \ \sum_{1 \leq i_1 < \cdots < i_k \leq n}  \mathbf{det}\left(
 \begin{array}{ccc}
(i_1|i_1) & \ldots  & (i_1|i_k) \\
\vdots   &          & \vdots    \\
(i_k|i_1) & \ldots  & (i_k|i_k) \\
 \end{array}
\right) \in {\mathbb C}[M_{n,n}].
\end{align*}
Clearly,  $\mathbf{h}_k(n) \in {\mathbb C}[M_{n,n}]^{ad_{gl(n)}}$.

Notice that the polynomials $\mathbf{h}_k(n)$'s  appear as coefficients (in ${\mathbb C}[M_{n,n}]$)
of the characteristic polynomial:
$$
P_{M_{n,n}}(t) = det \big( tI - M_{n,n} \big) =
t^n + 	\sum_{i=1}^n \ (-1)^i \ \mathbf{h}_i(n) \ t^{n-i}.
$$

From (\ref{Koszul}), we have

\begin{proposition}\label{Koszul map}
$$
\mathcal{K} \big( \mathbf{K}_{\lambda}(n) \big) = (-1)^{\binom{|\lambda|}{2}} \ 
\mathbf{h}_{\lambda_1}(n)\mathbf{h}_{\lambda_2}(n) 
\cdots \mathbf{h}_{\lambda_p}(n), \quad |\lambda| = \sum_i \ \lambda_i.
$$
\end{proposition}

Proposition \ref{K basis}
implies (is actually equivalent to) 
the  well-known theorem 
for the algebra of invariants ${\mathbb C}[M_{n,n}]^{ad_{gl(n)}}$:
\begin{proposition} \label{First Fund}
$$
{\mathbb C}[M_{n,n}]^{ad_{gl(n)}} = 
{\mathbb C} \big[ \mathbf{h}_1(n), \mathbf{h}_2(n), \ldots, \mathbf{h}_n(n) \big].
$$
Moreover, the $\mathbf{h}_k(n)$'s are algebraically independent.
\end{proposition}
Proposition \ref{First Fund} is usually stated in terms of the
algebra ${\mathbb C}[M_{n,n}]^{GL(n)} = {\mathbb C}[M_{n,n}]^{ad_{gl(n)}}$, where
${\mathbb C}[M_{n,n}]^{GL(n)}$ is the subalgebra of invariants with respect to the \emph{conjugation
action} of the general linear group ${GL(n)}$ on ${\mathbb C}[M_{n,n}]$
(see, e.g.  \cite{Procesi-BR}).

\section{Proof of Theorem \ref{MAIN}}

\subsection{A  commutation identity for enveloping algebras of Lie superalgebras}

Let $(L = L_0 \oplus L_1, [\ , \ ])$ be a \emph{Lie superalgebra} over $\mathbb{C}$
(see, e.g. \cite{KAC1-BR}, \cite{Scheu-BR}), where $[\ , \ ]$ denotes the \emph{superbracket}
bilinear form.

Given $a \in L$, consider the linear operator $T_a$ from $U(L)$ to
itself defined by setting
$$
T_a(\textbf{N})
= a \ \textbf{N} - (-1)^{|a| |\textbf{N}|} \textbf{N}  \ a,
$$ 
for every $\textbf{N} \in U(L)$,  $\mathbb{Z}_2$-homogeneous of degree $|\textbf{N}|$.

We recall that $T_a$ is the unique (left) superderivation of $U(L)$,
$\mathbb{Z}_2$-homogeneous of degree $|a|$, such that
$$
T_a(b) = [a, b],
$$
for every $b \in L$.

Furthermore, given $a, b \in L = L_0 \oplus L_1$, from (super) skew-symmetry and the 
(super) Jacobi identity, it follows:
$$
T_a \circ T_b - (-1)^{|a| |b|} T_b \circ T_a = 
T_{[a, b]}.
$$
The Lie algebra representation
$$
Ad_{L} : L = L_0 \oplus L_1 \rightarrow End_{{\mathbb C}} \big[ \mathbf{U}(L) \big]
= End_{{\mathbb C}} \big[ \mathbf{U}(L) \big]_0 \oplus
End_{{\mathbb C}} \big[ \mathbf{U}(L) \big]_1 
$$
$$
e_a \mapsto T_a
$$
is the \emph{adjoint representation} of $U(L)$ on itself.

\begin{proposition}\label{thm perm}
\begin{multline*}
a_{i_1}a_{i_2} \cdots a_{i_m} \omega = \omega a_{i_1}a_{i_2} \cdots a_{i_m} (-1)^{|\omega|(|a_{i_1}|+|a_{i_2}|+ 
\cdots +|a_{i_m}|)} +
\\
+ \sum_{k=1}^m \ \sum_{\sigma(1) < \cdots < \sigma(k); \sigma(k+1) < \cdots < \sigma(m)} \
\big( ( T_{a_{i_{\sigma(1)}}} \dots T_{a_{i_{\sigma(k)}}}(\omega) ) \ a_{i_{\sigma(k+1)}} \cdots a_{i_{\sigma(m)}} \times
\\
\times sgn(a_{i_{\sigma(1)}} \dots a_{i_{\sigma(k)}};a_{i_{\sigma(k+1)}} \cdots a_{i_{\sigma(m)}}) \
(-1)^{|\omega| (|a_{i_{\sigma(k+1)}}|+ \cdots +|a_{i_{\sigma(m)}}|)} \big).
\end{multline*}
\end{proposition}

\begin{proof} By induction hypotesis,
\begin{multline*}
a_{i_1}(a_{i_2} \cdots a_{i_m}) \omega =
a_{i_1} \omega a_{i_2} \cdots a_{i_m} (-1)^{|\omega|(|a_{i_2}|+ \cdots +|a_{i_m}|)} +
\\
+ a_{i_1}  \sum_{h=2}^m \ {\sum_{{\tau(2)} < \cdots < \tau(h); \tau(h+1) < \cdots < \tau(m)}} 
\big( T_{a_{i_{\tau(2)}}} \dots T_{a_{i_{\tau(h)}}}(\omega) a_{i_{\tau(h+1)}} \cdots a_{i_{\tau(m)}} \times
\\
\times sgn(a_{i_{\tau(2)} } \cdots  a_{i_{\tau(h)}}; a_{i_{\tau(h+1)}}  \cdots  a_{i_{\tau(m)}})
(-1)^{|\omega| (|a_{i_{\tau(h+1)}}|+ \cdots + \cdots |a_{i_{\tau(m)}}|)} \big) =
\\
=  \omega a_{i_1}a_{i_2} \cdots a_{i_m} (-1)^{|\omega|(|a_{i_1}|+|a_{i_2}|+ \cdots +|a_{i_m}|)} +
T_{a_{i_1}}(\omega) a_{i_2} \cdots a_{i_m} (-1)^{|\omega|(|a_{i_2}|+ \cdots +|a_{i_m}|)} +
\\
+ \sum_{h=2}^m \ \sum_{\tau(2) < \cdots < \tau(h); \tau(h+1) < \cdots < \tau(m)}
\big( T_{a_{i_1}}T_{a_{\tau(2)}} \cdots T_{a_{\tau(h)}}(\omega) a_{\tau(h+1)} \cdots a_{i_{\tau(m)}} \times
\\
\times sgn(a_{\tau(2)}  \cdots  a_{\tau(h)}; a_{\tau(h+1)}  \cdots  a_{\tau(m)})
(-1)^{|\omega| (|a_{\tau(h+1)}|+ \cdots + \cdots |a_{i_{\tau(m)}}|)} +
\\
+ T_{a_{\tau(2)}} \cdots T_{a_{\tau(h)}}(\omega) a_{i_1}a_{\tau(h+1)} \cdots a_{i_{\tau(m)}} \times
\\
(-1)^{|a_{i_1}|(|\omega|+|a_{\tau(2)}|+ \cdots +|a_{i_{\tau(m)}}|)}\times 
sgn(a_{\tau(2)}  \cdots  a_{\tau(h)}; a_{\tau(h+1)}  \cdots  a_{\tau(m)})
(-1)^{|\omega| (|a_{\tau(h+1)}|+ \cdots + \cdots |a_{i_{\tau(m)}}|)} \big),
\end{multline*}
where
\begin{multline*}
(-1)^{
|a_{i_1}|(|\omega|+|a_{i_{\tau(2)}}|+ \cdots +|a_{i_{\tau(m)}}|)
+
|\omega| (|a_{\tau(h+1)}|+ \cdots +  |a_{i_{\tau(m)}}|)
}
\times 
sgn(a_{i_{\tau(2)}} \cdots a_{i_{\tau(m)}}; a_{i_{\tau(h+1)}}  \cdots  a_{i_{\tau(m)}})
=
\\
= sgn(a_{i_{\tau(2)}}  \cdots  a_{i_{\tau(h)}}; a_{i_1}a_{i_{\tau(h+1)}}  \cdots  a_{i_{\tau(m)}})
(-1)^{|\omega| (|a_{i_1}|+|a_{i_{\tau(h+1)}}+ \cdots +|a_{i_{\tau(m)}}|)}.
\end{multline*}
Then, the assertion follows.
\end{proof}

In the Sweedler notation of the \emph{supersymmetric} superbialgebra $Super(L)$, Theorem \ref{thm perm}
can be stated in the following compact form:
\begin{proposition}\label{commutation Sweedler} Let
$$
\alpha = a_{i_1}a_{i_2} \cdots a_{i_m}.
$$
Then
$$
\alpha \omega = \sum_{(\alpha)} \ T_{\alpha_{(1)}}(\omega)\alpha_{(2)}(-1)^{|\omega||\alpha_{(2)}|}.
$$
\end{proposition}

\begin{proof}Let
$$
\alpha = a_{i_1}a_{i_2} \cdots a_{i_m}.
$$
Then, the coproduct (in the Sweedler  notation)
$$
\Delta(\alpha) = \sum_{(\alpha)} \ \alpha_{(1)} \otimes \alpha_{(2)}
$$
equals
$$
\sum_{k=0}^m \ \sum_{\sigma(1) < \cdots < \sigma(k); \sigma(k+1) < \cdots < \sigma(m)} \
\big( a_{i_{\sigma(1)}} \dots a_{i_{\sigma(k)}}  \otimes a_{i_{\sigma(k+1)}} \cdots a_{i_{\sigma(m)}} \times
$$
$$
\times sgn(a_{i_{\sigma(1)}} \cdots a_{i_{\sigma(k)}};a_{i_{\sigma(k+1)}} \cdots a_{i_{\sigma(m)}}) \big).
$$
\end{proof}

Furthermore
\begin{lemma} Let $T_\alpha = T_{a_1}T_{a_2} \cdots T_{a_m}$. Then
$$
T_\alpha(\omega_1 \cdot \omega_2) = 
\sum_{(\alpha)} \ T_{\alpha_{(1)}}(\omega_1)T_{\alpha_{(2)}}(\omega_2) (-1)^{|\alpha_{(2)}||\omega_1|}.
$$

\end{lemma}

\subsection{Some preliminary remarks and definitions}\label{preliminary lemmas}

\subsubsection{The virtual algebra and the Capelli devirtualization epimorphism}

Given a vector space $V$ of dimension $n$, we will regard it as a subspace of a $
\mathbb{Z}_2-$graded vector space
 $V_0 \oplus V_1$, where
$
V_1 = V.
$
The  vector spaces
$V_0$  (we assume that 
$dim(V_0)=m$ is ``sufficiently large'') is called
the  {\textit{positive virtual (auxiliary)
vector space}} and $V$ 
is called the {\textit{(negative) proper vector space}}.

Let
$
\mathcal{A}_0 = \{ \alpha_1, \ldots, \alpha_{m_0} \},$  
$\mathcal{L} = \{ 1, 2,  \ldots, n \}$
denote \emph{fixed  bases} of $V_0$ and $V = V_1$, respectively; 
therefore $|\alpha_s| = 0 \in \mathbb{Z}_2,$
and $ |i|   = 1 \in \mathbb{Z}_2.$

Let
$$
\{ e_{a, b}; a, b \in \mathcal{A}_0 \cup \mathcal{L}  \}, \qquad |e_{a, b}| =
|a|+|b| \in \mathbb{Z}_2
$$
be the standard $\mathbb{Z}_2-$homogeneous basis of the Lie superalgebra $gl(m|n)$ provided by the
elementary matrices. The elements $e_{a, b} \in gl(m|n)$ are $\mathbb{Z}_2-$homogeneous of
$\mathbb{Z}_2-$degree $|e_{a, b}| = |a| + |b|.$

The superbracket of the Lie superalgebra $gl(m|n)$ has the following explicit form:
$$
\left[ e_{a, b}, e_{c, d} \right] = \delta_{bc} \ e_{a, d} - (-1)^{(|a|+|b|)(|c|+|d|)} \delta_{ad}  \ e_{c, b},
$$
$a, b, c, d \in \mathcal{A}_0  \cup \mathcal{L} .$

In the following, the elements of the sets $\mathcal{A}_0,  \mathcal{L} $ will be called
\emph{positive virtual symbols} and \emph{negative proper symbols},
respectively.

 The inclusion $V \subset V_0 \oplus V_1$ induces a natural embedding of the ordinary general 
linear Lie algebra $gl(n) = gl(0|n)$ of $V$ into the
 {\textit{auxiliary}}
general linear Lie {\it{superalgebra}} $gl(m|n)$ of $V_0 \oplus V_1$ (see, e.g. \cite{KAC1-BR}, 
\cite{Scheu-BR}) and, hence, a natural embedding $\mathbf{U}(gl(n)) \subset \mathbf{U}(gl(m|n))$.

In the following, we will systematically refer to the
{\textit{Capelli devirtualization epimorphism}} 
$$
\mathfrak{p} : Virt(m,n)  \twoheadrightarrow 
\mathbf{U}(gl(0|n)) = \mathbf{U}(gl(n)),
$$
where $Virt(m,n)$ is the \emph{virtual subalgebra}
of $\mathbf{U}(gl(m|n))$.

For definitions and details, we refer the reader to Subsection \ref{citazione 4}.

\subsubsection{A more readable notation}

In the following,   we will adopt
the  more readable notation:

\begin{enumerate}
\item [--] We write $\{ a | b \}$ for the elements $e_{a, b}$ of the standard basis of $gl(m|n)$.

\item [--] Given two words $I = i_1 \ i_2 \ \cdots \ i_p$, $J= j_1 \ j_2 \ \cdots \ j_p$, with
$i_h , j_h \in  \mathcal{L}$ and a virtual symbol $\alpha$,
we write
$$
\{ J| \alpha \} = \{ j_1 \ j_2 \ \cdots \ j_p| \alpha \}, \quad 
\{ \alpha| I \} =   \{  \alpha | i_1 \ i_2 \ \cdots \ i_p \}
$$
in place of
$$
e_{j_1, \alpha}e_{j_2, \alpha} \cdots e_{j_p, \alpha},
\quad
e_{\alpha, i_1}e_{\alpha, i_2} \cdots e_{\alpha, i_p},
$$
respectively.
\end{enumerate}
In this notation, given a pair of Young tableaux 
$$S = (w_1, w_2, \ldots, w_p), \quad
T = (\overline{w}_1,  \overline{w}_2, \ldots,  \overline{w}_p), \qquad sh(S) = sh(T) = \lambda,
$$
the \emph{Capelli bitableau}
$$
[S | T] = \mathfrak{p} \big( e_{S C_\lambda} \cdot e_{ C_\lambda T} \big) \in \mathbf{U}(gl(n))
$$
is
\begin{equation*}\label{Capelli betableax notation}
[S | T] = \mathfrak{p} \big( \mathbf{P}_S \cdot \mathbf{P}_T \big),
\end{equation*}
where
$$
\mathbf{P}_S =  \{ w_1 | \beta_1 \} \{ w_2 | \beta_2 \} \cdots \{ w_p | \beta_p \},  \qquad          
\mathbf{P}_T =  \{ \beta_1 | \overline{w}_1 \}  \{ \beta_2 | \overline{w}_2 \} \cdots \{ \beta_p | \overline{w}_p \}. 
$$

Furthermore, for the adjoint representation
$$
Ad_{gl(m|n)} : gl(m|n) \rightarrow End_{{\mathbb C}} \big[ \mathbf{U}(gl(m|n)) \big]
$$
we write
\begin{enumerate}

\item [--]
$
T_{i \alpha}, \ T_{\alpha i}
$
\ in place of \
$
T{e_{i \alpha}}, \ T_{e_{\alpha i}}.
$

\item [--] 
$
T_{I  \alpha}, 
\
T_{ \alpha  I}
$
\ in place of \
$
T_{i_1, \alpha}T_{i_2, \alpha} \cdots T_{i_p, \alpha},
\quad
T_{\alpha, i_1}T_{\alpha, i_2} \cdots T_{\alpha, i_p},
$
respectively.

\end{enumerate}

\subsubsection{The coproduct in $\Lambda(V) = \Lambda(\mathcal{L})$, Sweedler notation and \emph{split notation}}

Given a word $I = i_1 \ i_2 \ \ \cdots \ i_m, \ i_t \in \mathcal{L}$ in $\Lambda(V) = \Lambda(\mathcal{L})$, 
and a natural integer $k, \quad k = 0, 1, \cdots, m$, consider the homogeneous 
component
$$
\Delta_{k, m-k} : \Lambda(\mathcal{L}) \rightarrow \Lambda(\mathcal{L})_k \otimes \Lambda(\mathcal{L})_{m-k}
$$ 
of the coproduct
$$
\Delta : \Lambda(\mathcal{L}) \rightarrow \Lambda(\mathcal{L}) \otimes \Lambda(\mathcal{L}).
$$ 
Given a permutation $\sigma$
with
$$
\sigma(1) < \cdots < \sigma(k), \quad \sigma(k+1) <  \cdots < \sigma(m),
$$
and the two subwords
$$
I_{(1)} = i_{\sigma(1)} \ \ \cdots \ i_{\sigma(k)}, \quad I_{(2)} = i_{\sigma(k+1)} \ \ \cdots \ i_{\sigma(m)}
$$
we call the pair $(I_{(1)}, I_{(2)})$ a \emph{split} of $I$ of step $(k, m-k)$
of signature $sgn(I; I_{(1)}, I_{(2)}) = sgn(\sigma)$.
Clearly, $I = sgn(I; I_{(1)}, I_{(2)}) \ I_{(1)} I_{(2)}$.

We denote by $\mathbf{S}(I; k,m-k)$ the set of all splits  of $I$ of step $(k, m-k)$.

Then, the coproduct component
$$
\Delta_{k, m-k}(I) = \sum_{(I)_{k,m-k}} \ I_{(1)} \otimes I_{(2)}
$$
can be explicitly written as
$$
\Delta_{k, m-k}(I) = \sum_{(I_{(1)}, I_{(2)}) \in \mathbf{S}(I; k,m-k)}
\
sgn(I; I_{(1)}, I_{(2)}) \ I_{(1)} \otimes I_{(2)}.
$$

\subsection{Some lemmas}

Consider the Capelli bitableau
$$
[ S | T ] = \mathfrak{p} \big( \mathbf{P}_S \cdot \mathbf{P}_T \big)
$$
as in Eq. (\ref{Capelli betableax notation}).

From Proposition \ref{commutation Sweedler}, we derive the following pair of Lemmas.

\begin{lemma}\label{lemma 1pre} Let $I = i_1 \ i_2 \cdots \ i_m$, $J= j_1 \ j_2 \ \cdots \ j_m$,
$m \leq \lambda_p$. 

Then
$$
\{ J| \alpha\} \{ \alpha | I \} \ \emph{\textbf{P}}_S
$$
equals
$$
\{ J| \alpha\} \ \sum_{k=0}^m \ \sum_{(I)_{k, m-k}} \
T_{\alpha \ I_{(1)}} \big( \emph{\textbf{P}}_S \big) \{ \alpha |I_{(2)} \} (-1)^{|\emph{\textbf{P}}_S|(m-k)}. 
$$
\end{lemma}

Since
$$
\mathfrak{p} \big( \{ J| \alpha\} \{ \alpha | I \} \ \emph{\textbf{P}}_S \cdot \emph{\textbf{P}}_T \big)
= [J | I ] \ [ S | T ],
$$

\begin{lemma}\label{assorbimenti 1} We have

\begin{multline}\label{double sum}
[J | I ] \ [ S | T ] = \ 
(-1)^{(|\emph{\textbf{P}}_T|+k)(m-k)} \times
\\
\times \mathfrak{p} \big(
\sum_{k=0}^m \ \sum_{(I)_{k, m-k}} \  \sum_{(J)_{k, m-k}} 
T_{ J_{(1)} \alpha } 
T_{\alpha  I_{(1)}} \big( \emph{\textbf{P}}_S \big) \ \{ J_{(2)}| \alpha\} 	  
 \  \emph{\textbf{P}}_T \ \{\alpha | I_{(2)} \} \big).
\end{multline}

\end{lemma}
\begin{proof} We have
\begin{multline*}
\{ J| \alpha\} \{ \alpha | I \} \ \emph{\textbf{P}}_S \ \emph{\textbf{P}}_T =
\\
= \{ J| \alpha\} \ \sum_{k=0}^m \ \sum_{(I)_{k, m-k}} \
T_{\alpha \ I_{(1)}} \big( \emph{\textbf{P}}_S \big) \ \{\alpha | I_{(2)} \} \ \emph{\textbf{P}}_T \ 
(-1)^{|\emph{\textbf{P}}_S |(m-k)} = 
\\
 = \sum_{k=0}^m \ \sum_{(I)_{k, m-k}} \{ J| \alpha\} \
T_{\alpha \ I_{(1)}} \big( \emph{\textbf{P}}_S  \big) \ \{\alpha | I_{(2)} \} \ \emph{\textbf{P}}_T \ 
(-1)^{|\emph{\textbf{P}}_S |(m-k)} =
\\
=  \sum_{k=0}^m \ \sum_{(I)_{k, m-k}} \ \left( \sum_{h=0}^m \sum_{(J)_{h, m-h}} 
T_{ J_{(1)} \alpha } \big(
T_{\alpha \ I_{(1)}} \big( \emph{\textbf{P}}_S  \big) \big)  \{J_{(2)}| \alpha \} \ (-1)^{(|\emph{\textbf{P}}_S|
+ h)(m-h)} \right) \times
\\
\times \{\alpha | I_{(2)} \} \ \emph{\textbf{P}}_T \
(-1)^{|\emph{\textbf{P}}_S |(m-k)}.
\end{multline*}

Now, if $h < k$, then $m-h > m-k$ and, hence,
\begin{multline*}
 \sum_{(I)_{k, m-k}} \ \left( \sum_{(J)_{h, m-h}} 
T_{ J_{(1)} \alpha } \big(
T_{\alpha \ I_{(1)}} \big( \emph{\textbf{P}}_S  \big) \big)  \{J_{(2)}| \alpha \} \ (-1)^{(|\emph{\textbf{P}}_S|
+ h)(m-h)} \right) \times
\\
\times \{\alpha | I_{(2)} \} \ \emph{\textbf{P}}_T \
(-1)^{|\emph{\textbf{P}}_S |(m-k)}
\end{multline*}
is an \emph{irregular element}, 
since the $\{J_{(2)}| \alpha \} \{\alpha | I_{(2)} \}$ are irregular monomials;
so, its image with respect to the Capelli epimorphism $\mathfrak{p}$ equals zero.

If $h > k$, then,
$$
T_{ J_{(1)} \alpha } \big(
T_{\alpha \ I_{(1)}} \big( \emph{\textbf{P}}_S  \big) \big) = 0.
$$
and, hence,
\begin{multline*}
 \sum_{(I)_{k, m-k}} \ \left( \sum_{(J)_{h, m-h}} 
T_{ J_{(1)} \alpha } \big(
T_{\alpha \ I_{(1)}} \big( \emph{\textbf{P}}_S  \big) \big)  \{ J_{(2)}| \alpha \} \ (-1)^{(|\emph{\textbf{P}}_S|
+ h)(m-h)} \right) \times
\\
\times \{\alpha | I_{(2)} \} \ \emph{\textbf{P}}_T \
(-1)^{|\emph{\textbf{P}}_S |(m-k)} = 0.
\end{multline*}

Then,
\begin{align*}
[J | I ] \ [ S | T ] = \ &(-1)^{(|\emph{\textbf{P}}_S|+k)(m-k)} (-1)^{|\emph{\textbf{P}}_S|(m-k)} \times
\\
&\times \mathfrak{p} \big(
\sum_{k=0}^m \ \sum_{(I)_{k, m-k}} \  \sum_{(J)_{k, m-k}} 
T_{ J_{(1)} \alpha } 
T_{\alpha  I_{(1)}} \big( \emph{\textbf{P}}_S \big) \ \{ J_{(2)}| \alpha\} 	  
 \  \{\alpha | I_{(2)}\}  \ \emph{\textbf{P}}_T  \big)
\\
= \
&(-1)^{(|\emph{\textbf{P}}_T|+k)(m-k)} \times
\\
&\times \mathfrak{p} \big(
\sum_{k=0}^m \ \sum_{(I)_{k, m-k}} \  \sum_{(J)_{k, m-k}} 
T_{ J_{(1)} \alpha } 
T_{\alpha  I_{(1)}} \big( \emph{\textbf{P}}_S \big) \ \{ J_{(2)}| \alpha\} 	  
 \  \emph{\textbf{P}}_T \ \{\alpha | I_{(2)} \} \big).
\end{align*}

\end{proof}

\begin{corollary}
Let $m \leq \lambda_p$. Then
$$
[J | I ] \ [ S | T ] = 
\pm \left[
\begin{array}{lll}
S

\\
J
\end{array} 
\right|\left.
\begin{array}{lll}
T
\\
I
\end{array}
\right]
+
\sum  c_{m, \lambda} \ \mathbf{G}_{m, \lambda},
$$
where
$$
[J | I ] \ [ S | T ], \quad 
\left[
\begin{array}{lll}
S
\\

J
\end{array} 
\right|\left.
\begin{array}{lll}
T
\\
I
\end{array}
\right] \notin \mathbf{U}(gl(n))^{(n)}
\quad
whenever 
\quad
 n < m + |\lambda|,
$$
and
$$
\mathbf{G}_{m, \lambda} \in \mathbf{U}(gl(n))^{(n)}
\quad
for \ some
\quad
 n < m + |\lambda|.
$$
\end{corollary}

\begin{corollary}\label{filtration corollary}
Let $m \leq \lambda_p$. Then
$$
 [ S | T ] = 
\pm 
\ [ \omega_1 | \overline{\omega}_1 ] \ [ \omega_2 | \overline{\omega}_2 ] \cdots [ \omega_p | \overline{\omega}_p ]
+
\sum  d_{\lambda} \ \mathbf{F}_{\lambda},
$$
where
$$
[ S | T ], \
[ \omega_1 | \overline{\omega}_1 ] \ [ \omega_2 | \overline{\omega}_2 ] \cdots [ \omega_p | \overline{\omega}_p ]
\notin \mathbf{U}(gl(n))^{(n)}
\quad
whenever 
\quad
 n < |\lambda|,
$$
and
$$
\mathbf{F}_{ \lambda} \in \mathbf{U}(gl(n))^{(n)}
\quad
for \ some
\quad
 n < |\lambda|.
$$
\end{corollary}

We specialize the previous results to Capelli-Deruyts bitableaux $\mathbf{K}^\lambda$.

Let
$$
\textbf{M}^* = \{\underline{\lambda_1}^* | \beta_1   \} 
\cdots  \{\underline{\lambda_p}^* | \beta_p   \}, \quad
\textbf{M} =  \{ \beta_1 | \underline{\lambda_1}  \} \cdots \{ \beta_p | \underline{\lambda_p} \}, 
$$
where
 $\lambda = (\lambda_1 \geq \cdots \geq \lambda_p)$ and
$|\textbf{M}^*| = |\textbf{M}| = |\lambda| = \lambda_1+ \cdots + \lambda_p \in \mathbb{Z}_2$.

Given an increasing word $W = h_1 \ h_2 \ \cdots \ h_p$ on 
$\mathcal{L} = \{1, 2, \ldots, n \}$, denote by $W^*$ its \emph{reverse} word, that is:
$$
W^* = h_p \ \cdots \ h_2 \ h_1.
$$

Let $I = 1 \ 2 \ \cdots \ m$, $I^* = m \ m-1 \ \cdots \  1$, $m \leq \lambda_p$.

In this notation
$$
\mathbf{K}^\lambda = \mathfrak{p} \big( \textbf{M}^* \cdot \textbf{M} \big)
$$
and
$$
[ I^* | I ] \ \mathbf{K}^\lambda = 
\mathfrak{p} \big(\ \{ I^*| \alpha\} \{ \alpha | I \} \ \textbf{M}^* \cdot \textbf{M} \ \big).
$$

We apply Lemma \ref{assorbimenti 1} to the element $[ I^* | I ] \ \mathbf{K}^\lambda$.
As we shall see, the double  sum 
$$
\sum_{(I^*)_{k, m-k}}  \ \sum_{(I)_{k, m-k}}
$$ 
in eq. (\ref{double sum}) reduces to a single sum
$$
\sum_{(I)_{k, m-k}}
$$
since the only splits $I^*_{(1)}, \ I^*_{(2)}$ in $(I^*)_{k, m-k}$ that give rise to nonzero 
summands are those for
$$
I^*_{(1)} = (I_{(1)})^* \quad \text{and} \quad I^*_{(2)} = (I_{(2)})^*,
$$
where $(I_{(1)})^*, \ (I_{(2)})^*$ are the reverse words of $I_{(1)}$ and $I_{(2)}$,
respectively.

\begin{lemma}\label{MAIN LEMMA} The element
$$
[ I^* | I ] \ \mathbf{K}^\lambda = \mathfrak{p} \big( \{ I^*| \alpha\} \{ \alpha | I \} \ \emph{\textbf{M}}^* 
\cdot \emph{\textbf{M}} \big) 
$$
equals
\begin{multline*}
\sum_{k=0}^m \ (-1)^{(|\textbf{M}|+k)(m-k)}
\sum_{(I)_{k, m-k}} \ 
\mathfrak{p} \big( 
T_{ (I_{(1)})^* \alpha } \big(
T_{\alpha \ I_{(1)}} \big( \emph{\textbf{M}}^* 	\big) \big)  \{ {(I_{(2)})^*}| \alpha\} 
 \emph{\textbf{M}}   \{\alpha | I_{(2)} \} \big).
\end{multline*}
\end{lemma}
\begin{proof}

From Lemma \ref{assorbimenti 1}, we have
\begin{multline*}
\mathfrak{p} \big( \{ I^*| \alpha\} \{ \alpha | I \} \ \textbf{M}^* \cdot \textbf{M} \big) 
=
\sum_{k=0}^m \ (-1)^{(|\textbf{M}|+k)(m-k)}
\\ \big(\sum_{(I)_{k, m-k}}  \ \sum_{(I^*)_{k, m-k}} \ \mathfrak{p} \big( T_{I^*_{(1)}\alpha } \big(
T_{\alpha \ I_{(1)}} \big( \textbf{M}^* \big) 	\big) \ \{ (I^*_{(2)}| \alpha\} \ 
 \textbf{M} \ \{\alpha | I_{(2)} \} \ \big) \big).
\end{multline*}

Let $k = 0, 1, \ldots, m$ and examine
the element
\begin{multline*}
\sum_{(I)_{k, m-k}} \ \sum_{(I^*)_{k, m-k}} 
T_{ {I^*_{(1)} \alpha }} \big(
T_{\alpha \ I_{(1)}} \big( \textbf{M}^* \big) \big) \{ I^*_{(2)}| \alpha\}  
 \ \textbf{M} \{\alpha | I_{(2)} \} =
\\
= \sum_{(I)_{k, m-k}} \ \sum_{(I^*)_{k, m-k}} 
T_{ {I^*_{(1)} \alpha }} \big(
T_{\alpha \ I_{(1)}} \big(\{\underline{\lambda_1}^* | \beta_1   \} 
\cdots  \{\underline{\lambda_p}^* | \beta_p   \} \big) \big) \{ {(I_{(2)})^*}| \alpha\}  
\ \textbf{M} \{\alpha | I_{(2)} \} .
\end{multline*}
If $i \in I_{(2)}$, then $i \notin I_{(1)}$. Hence, all the variables
$$
\{ i |\beta_q \} \qquad q = 1, 2, \ldots ,p
$$ appear in
$$
T_{\alpha,I_{(1)}}\big( \{\underline{\lambda_1}^* | \beta_1   \} 
\cdots  \{\underline{\lambda_p}^* | \beta_p   \}\big),
$$
for every $q = 1, 2, \ldots, p$.

Assume that $i \notin I^*_{(2)}$, then $i \in I^*_{(1)}$.
Hence, $\exists \ \underline{q} \in \{1, 2, \ldots, p \}$ such that the variable
$$
\{ i |\beta_{\underline{q}}  \}
$$
is \textit{created} by the action of
$$
T_{ I^*_{(1)} \alpha }
$$
on
$$
T_{\alpha,I_{(1)}}\big( \{\underline{\lambda_1}^* | \beta_1   \} 
\cdots  \{\underline{\lambda_p}^* | \beta_p   \}\big) \qquad (*).
$$
Then $(*)$ contains two occurrencies of $\{ i | \beta_{\underline{q}}  \}$
and, hence, \emph{equals zero}.
Therefore
$$
T_{ I^*_{(1)} \alpha } T_{\alpha,I_{(1)}}\big( \{\underline{\lambda_1}^*  | \beta_1   \} 
\cdots  \{\underline{\lambda_p}^*  | \beta_p   \}\big)  \neq 0
$$
implies
\begin{equation*}\label{not zero}
i \in I_{(2)} \implies i \in I^*_{(2)}.
\end{equation*}

Since $I_{(2)}$ and $I^*_{(2)}$ are words of the same length $m-k$,
this implies that
 the only  \emph{not zero} summands 
- \emph{with respect to the action of the Capelli epimorphism} $\mathfrak{p}$ - in
$$
\sum_{(I)_{k, m-k}} \ \sum_{(I^*)_{k, m-k}} 
\mathfrak{p} \big(
T_{ {I^*_{(1)} \alpha }} \big(
T_{\alpha \ I_{(1)}} \big( \textbf{M}^* \big) \big) \{  I^*_{(2)}| \alpha\}   
 \ \textbf{M} \{\alpha | I_{(2)} \} \big)
$$
are for $I^*_{(1)} = (I_{(1)})^*$ and $I^*_{(2)} = (I_{(2)})^*$, 
that is
$$
\mathfrak{p} \big( T_{ (I_{(1)})^* \alpha } \big(
T_{\alpha \ I_{(1)}} \big( \textbf{M}^* \big) 	\big) \ \{ {(I_{(2)})^*}| \alpha\}  
 \ \textbf{M} \{\alpha | I_{(2)} \} \big).
$$
\end{proof}


Let us examine the expression
\begin{equation}\label{expression} 
\sum_{(I)_{k.m-k}}
(-1)^{k (m - k)} \ T_{ (I_{(1)})^* \alpha } \big(
T_{\alpha \ I_{(1)}} \big( \textbf{M}^* \big) \big) 	 \ \{ {(I_{(2)})^*}| \alpha\}  \{\alpha | I_{(2)} \}.
\end{equation}
in the notation of \emph{splits}.

\begin{corollary} The expression (\ref{expression}) equals
\begin{equation*}\label{expression three}
\sum_{(A, B) \in S(I; k, m-k)} \
T_{ A^* \alpha } \big(
T_{\alpha  A} \big( \emph{\textbf{M}}^* \big) 	\big) \ \{B^* | \alpha\}  \{\alpha | B \}.
\end{equation*}
\end{corollary}
\begin{proof} In the notation of \emph{splits}, the expression (\ref{expression}) equals
\begin{multline*}\label{expression two}
(-1)^{k (m - k)} \ \sum_{(A, B) \in S(I; k, m-k)}
T_{ A^* \alpha } \big(
T_{\alpha \ A} \big( \emph{\textbf{M}}^* \big) 	\big) \ \{B^*| \alpha\}  \{\alpha | B \} \times 
\\
sgn(I; A, B)sgn(I^*; A^*, B^*).
\end{multline*}

We have
\begin{multline*}
(-1)^{k (m - k)}sgn(I; A, B)sgn(I^*; A^*, B^*) =
\\
= (-1)^{k (m - k)}(-1)^{k (m - k)}
sgn(I; A, B)sgn(I^*; B^*, A^*). 
\end{multline*}
But $sgn(I; A, B)sgn(I^*; B^*, A^*) = 1$.
\end{proof}

Given  $(A, B) \in S(I; k, m-k)$, let $A = a_1  a_2  \cdots  a_k$,
$\{ a_1 < a_2 < \cdots < a_k \}   \subseteq \{ 1, 2, \ldots, m \}$ and recall
$$
\textbf{M}^* = \{\underline{\lambda}_1|\beta_1 \} 
\cdots  \{\underline{\lambda_p}^*|\beta_p \};
$$
we examine the element
\begin{equation}\label{element}
T_{ A^* \alpha } T_{\alpha \ A} \big( \textbf{M}^* \big).
\end{equation}

\begin{lemma}\label{raising factorial} We have
$$
T_{ A^* \alpha } T_{\alpha \ A} \big( \emph{\textbf{M}}^* \big) =
{{\langle}  \ {p} {\rangle}_k} \ \{ \underline{\lambda_1}^* | \beta_{1} \} 
\cdots
\{\underline{\lambda_p}^*| \beta_{p} \} = \ {{\langle}  \ {p} {\rangle}_k} \ \emph{\textbf{M}}^* ,
$$
where
$$
{{\langle}  \ {p} {\rangle}_k} = p(p+1) \cdots (p+k-1)
$$
is the \emph{raising factorial} coefficient.
\end{lemma}

\begin{proof} By skew-symmetry, a simple computation shows
that (\ref{element}) equals
\begin{equation}\label{equation}
\sum_{h_1 + \cdots + h_p = k} \ \sum_{(A_1, \ldots, A_p) \in S(A; h_1, \ldots, h_p)} \
T_{(A_1)^* \alpha }T_{\alpha A_1}
\big( \{\underline{\lambda_1}^* | \beta_{1} \} \big) 
\cdots
T_{(A_p)^* \alpha }T_{\alpha A_p}
\big( \{\underline{\lambda_p}^* | \beta_{p} \} \big).
\end{equation}
We examine the value of
\begin{equation*}\label{element bis}
T_{C^* \alpha }T_{\alpha C}
\big( \{\underline{q}^* | \beta \} \big)
\end{equation*}
for $C = c_1c_2 \cdots c_h$, $\{c_1 < c_2 < \ldots <c_h \} \subseteq \{1, 2, \ldots q \}$.

Clearly
$$ 
\{\underline{q}^* | \beta \} =
 \{\underline{q} | \beta \}  (-1)^{\binom {q}{2}},
$$
and a simple computation shows that
$$
T_{C^* \alpha }T_{\alpha C}
\big( \{\underline{q} | \beta \} \big) = h! \ \{\underline{q} | \beta \}. 
$$
Indeed, we have
\begin{align*}
T_{\alpha C}
\big( \{\underline{q} | \beta \} \big) &= 
T_{c_1 \alpha} \cdots T_{c_h \alpha} \big( \{1 | \beta \} \cdots \{q | \beta \} \big)
\\
&= \{1 | \beta \} \cdots \widehat{ \{c_1 | \beta \}} \{ \alpha | \beta \}
\cdots \widehat{ \{c_h | \beta \}} \{ \alpha | \beta \} \cdots \{q | \beta \} \
(-1)^{c_h-1 + \cdots + c_1-1}
\\
& = \{ \alpha | \beta \}^h \{1 | \beta \} \cdots \widehat{ \{c_1 | \beta \}} 
\cdots \widehat{ \{c_h | \beta \}}  \cdots \{q | \beta \} \
(-1)^{c_h-1 + \cdots + c_1-1};
\end{align*}
now, 
\begin{align*}
T_{C \alpha } T_{\alpha C}
\big( \{\underline{q} | \beta \} \big) &=
 T_{c_h \alpha} \cdots T_{c_1 \alpha} 
\big( \{ \alpha | \beta \}^h \{1 | \beta \} \cdots \widehat{ \{c_1 | \beta \}} 
\cdots  \big)
(-1)^{c_h-1 + \cdots + c_1-1}
\\
& = h! \{c_h | \beta \} 	\cdots  \{c_1 | \beta \} \cdots \widehat{ \{c_1 | \beta \}} \cdots
 \widehat{ \{c_h | \beta \}} (-1)^{c_h-1 + \cdots + c_1-1}
\\
& = h! \{1 | \beta \} \cdots \{q | \beta \} = h! \{\underline{q} | \beta \}.
\end{align*}

Then,
$$
T_{C^* \alpha }T_{\alpha C}
\big( \{\underline{q}^* | \beta \} \big)
= (-1)^{ \binom{q}{2}} T_{C^* \alpha }T_{\alpha C}
\big( \{\underline{q} | \beta \} \big) =
(-1)^{ \binom{q}{2}}  \ h! \ \{\underline{q} | \beta \} =
h! \ \{\underline{q}^* | \beta \}.
$$

Hence,  (\ref{equation}) equals
\begin{multline*}
\sum_{(h_1, \ldots, h_p); h_1 + \cdots + h_p = k} \ \sum_{(A_1, \ldots, A_p) \in S(A; h_1, \ldots, h_p)} \
h_1! \cdots h_p! \
\big( \{\underline{\lambda_1}^* | \beta_{1} \} 
\cdots
\{\underline{\lambda_p}^* | \beta_{p} \} \big) =
\\
= \sum_{ h_1 + \cdots + h_p = k} \ \frac{k!} {h_1! \cdots h_p!} \ h_1! \cdots h_p! \
\big( \{\underline{\lambda_1}^* | \beta_{1} \} 
\cdots
\{\underline{\lambda_p}^* | \beta_{p} \} \big) 
\end{multline*}
that equals
\begin{equation*}
\abinom{p}{k} \ k! \ \big( \{\underline{\lambda_1}^*| \beta_{1} \} 
\cdots
\{\underline{\lambda_p}^* | \beta_{p} \} \big) = 
{{\langle}  \ {p} {\rangle}_k} \ \{\underline{\lambda_1}^* | \beta_{1} \} 
\cdots
\{\underline{\lambda_p}^* | \beta_{p} \}.
\end{equation*}
\end{proof}

Hence, from Lemma \ref{MAIN LEMMA} and Lemma \ref{raising factorial}, we infer:
\begin{proposition}\label{final prop} Let $I = 1 2 \cdots m$, $I^* = m  \cdots  2 1$.
Then
\begin{align*}
[ I^* | I ] \ \mathbf{K}^\lambda &= 
\mathfrak{p} \big(\ \{ I^*| \alpha\} \{ \alpha | I \} \ \textbf{M}^* \cdot \textbf{M} \ \big) 
\\
&= \mathfrak{p} \big( \{ I^*| \alpha\} \{ \alpha | I \} \ \{\underline{\lambda_1}^* | \beta_1   \} 
\cdots  \{\underline{\lambda_p}^* | \beta_p   \} \{ \beta_{1} |  \underline{\lambda}_{1} \}
\cdots 
\{ \beta_{p} |  \underline{\lambda}_{p} \} \big) 
\end{align*}
equals
$$
\sum_{k=0}^m \ (-1)^{|\textbf{M}|(m-k)}
\sum_{(A, B) \in S(I; k, m-k)}
{{\langle}  \ {p} {\rangle}_k} \ \mathfrak{p} \big( \textbf{M}^* \{B^* | \alpha\}  \textbf{M} \{\alpha | B \} \big).
$$
\end{proposition}

\subsection{Proof of Theorem \ref{MAIN}} 
Let  $m \leq \lambda_p$ and  $M \subseteq \underline{\lambda_p}$, $|M| = m$, 
as in Theorem \ref{MAIN}.

Recall that $|\textbf{M}| = |\textbf{M}^*| =  |\lambda| \in \mathbb{Z}_2$,
where $|\lambda| = \lambda_1 + \cdots + \lambda_p$.

From Remark \ref{Deruyts remark} and 
Proposition \ref{final prop}, we have:
\begin{align*}
\left[ M^* | M \right] \mathbf{K}^\lambda &=
 \mathfrak{p} \left(\{M^* |\alpha \} \{\alpha | M \} 
 \textbf{M}^* \cdot \textbf{M}
\right) 
\\
&=
\sum_{k = 0}^m
\
\left< p \right>_{m-k}
\
(-1)^{|\lambda| k}
\sum_{J;\ J \subseteq M;\ |J|=k}
\mathfrak{p}
\left(
\textbf{M}^* 
\{ J^* | \alpha \} 
\textbf{M}
\{ \alpha | J \} \right) 
\\
&
\stackrel{def}{=}
\sum_{k = 0}^m
\
\left< p \right>_{m-k} \ (-1)^{|\lambda|k}
\
\sum_{J;\ J \subseteq M;\ |J|=k}
\left[
\begin{array}{l}
\mathbf{K}^\lambda
\\
J
\end{array}
\right].
\end{align*}

\section{ Proof of Theorem \ref{The hook coefficient lemma-BR}}

\begin{proof} Recall that
$$
v_{\mu} = (Der_{\tilde{\mu}} | Der^P_{\tilde{\mu}}),
$$
where $(Der_{\tilde{\mu}} | Der^P_{\tilde{\mu}})$ is the Young bitableau 
(see, e.g. Subsection \ref{citazione 6} below)
$$
\left(
\begin{array}{lllll}
1 \ \ 2 \ \ \cdots \ \cdots \ \ \tilde{\mu}_1 \\
\\
1 \ \ 2  \ \ \cdots \  \tilde{\mu}_2
\\
\cdots  \ \cdots \ \cdots
\\
\cdots  \ \cdots \ 
\\
1 \ \ 2  \ \ \tilde{\mu}_q
\end{array}
\right| \left.
\begin{array}{lllll}
1 \ \ 2 \ \ \cdots \ \cdots \ \ \tilde{\mu}_1 \\
\\
1 \ \ 2  \ \ \cdots \  \tilde{\mu}_2
\\
\cdots  \ \cdots \ \cdots
\\
\cdots  \ \cdots \ 
\\
1 \ \ 2  \ \ \tilde{\mu}_q
\end{array}
\right)
$$
in the polynomial algebra ${\mathbb C}[M_{n,d}]$. 

Set
$$ e_{Der^*_{n^p}, Coder_{n^p}} \ = \
e_{n \alpha_1 } \cdots e_{1 \alpha_1 } \cdots \cdots
e_{n \alpha_{p-1}} \cdots e_{1 \alpha_{p-1}}e_{n \alpha_p } \cdots e_{1 \alpha_p}.
$$

Set
$$ e_{Coder_{n^p}, Der_{n^p}} \ = \
e_{\alpha_1 1} \cdots e_{\alpha_1 n} \cdots \cdots
e_{\alpha_{p-1} 1} \cdots e_{\alpha_{p-1} n}e_{\alpha_p 1} \cdots e_{\alpha_p n}.
$$

Since
$$
\mathbf{K_n^p} = \mathfrak{p}\big( e_{Der^*_{n^p}, Coder_{n^p}} \ e_{Coder_{n^p}, Der_{n^p}} \big),
$$
the action of $\mathbf{K_n^p}$ on $v_{\mu} = (Der_{\tilde{\mu}} | Der^P_{\tilde{\mu}})$ is the same as the
action of 
$$
e_{Der^*_{n^p}, Coder_{n^p}} \ e_{Coder_{n^p}, Der_{n^p}}.
$$

We follow \cite{Regonati-BR} (see Proposition 5).

Now, if $\mu_n = 0$, then
$$
e_{\alpha_p n} \cdot (Der_{\tilde{\mu}} | Der^P_{\tilde{\mu}})
$$
is zero. 

In the following, we limit ourselves to write the left parts 
of the Young bitableaux involved.

If $\mu_n \geq 1$, then
$$
e_{\alpha_p n} \cdot (Der_{\tilde{\mu}} | Der^P_{\tilde{\mu}})
$$
 equals
\begin{equation}\label{first action}
(-1)^{n-1}
\left(
\begin{array}{lllll}
1 \ \ 2 \ \ \cdots \ n-1 \ \ \alpha_p\\
\cdots  \ \cdots \ \cdots
\\
1 \ \ 2  \ \ \cdots \  n-1  \  \  \ n
\\
\cdots  \ \cdots \ \cdots
\\
1 \ \ 2  \ \ \cdots \  n-1 \  \  \ n
\\
1 \ \ 2 \ \ \cdots \ \cdots
\\
\cdots  \ \cdots \ \cdots
\end{array}
\right| 
+
\cdots
+
(-1)^{n (\mu_n -1) + n - 1}
\left(
\begin{array}{lllll}
1 \ \ 2 \ \ \cdots \ n-1 \ \ n\\
\cdots  \ \cdots \ \cdots
\\
1 \ \ 2  \ \ \cdots \  n-1 \ \ n
\\
\cdots  \ \cdots \ \cdots
\\
1 \ \ 2  \ \ \cdots \  n-1  \ \ \alpha_p
\\
1 \  \ 2 \ \cdots \ \cdots
\\
\cdots  \ \cdots \ 
\end{array}
\right|,
\end{equation}
by Proposition \ref{action on tableaux}.

A simple sign computation shows that (\ref{first action}) 
equals
$$(-1)^{n-1} \ \mu_n
(-1)^{n-1}
\left(
\begin{array}{lllll}
1 \ \ 2 \ \ \cdots \ n-1 \ \ \alpha_p\\
\cdots  \ \cdots \ \cdots
\\
1 \ \ 2  \ \ \cdots \  n-1  \  \  \ n
\\
\cdots  \ \cdots \ \cdots
\\
1 \ \ 2  \ \ \cdots \  n-1 \  \  \ n
\\
1 \ \ 2 \ \ \cdots \ \cdots
\\
\cdots  \ \cdots \ \cdots
\end{array}
\right|.
$$

Now, again by Proposition \ref{action on tableaux} and  simple computation, we have:
\begin{multline*}
e_{\alpha_p n-1} \cdot \left(
\begin{array}{lllll}
1 \ \ 2 \ \ \cdots \ n-1 \ \ \alpha_p
\\
1 \ \ 2  \ \ \cdots \  \ \cdots  \  \ \ \ n
\\
\cdots  \ \cdots \ \cdots
\\
\cdots  \cdots \cdots
\\
1 \ 2 \ \cdots \ \cdots
\\
\cdots  \ \cdots \ \cdots
\end{array}
\right| =
\\
= 
(-1)^{n-2}
\left(
\begin{array}{lllll}
1 \ \ 2 \ \ \cdots \ \ \alpha_p \ \ \ \alpha_p
\\
1 \ \ 2  \ \ \cdots \  \ \cdots  \  \  \ n
\\
\cdots  \ \cdots \ \cdots
\\
1 \ \ 2  \ \ \cdots \  \ \cdots  \   \  \ n
\\
1 \ 2 \ \cdots \ \cdots
\\
\cdots  \ \cdots \ \cdots
\end{array}
\right|
+
\\
+
\sum_{i = 2}^{\mu_n}
\
(-1)^{(n-1) + (i-2)n +(n-2)}  
\left(
\begin{array}{lllll}
1 \ \ 2 \ \ \cdots \ n-1 \ \ \ \alpha_p
\\
1 \ \ 2  \ \ \cdots \ n-1 \ \ \ n 
\\
\cdots  \ \cdots \ \cdots
\\
1 \ \ 2  \ \ \cdots \  \  \ \alpha_p  \ \ \ \ \ n
\\
\cdots  \ \cdots \ \cdots
\\
1 \ \ 2  \ \ \cdots \  \ \cdots  \  \  \  \ \ n
\\
1 \ \ 2 \ \  \cdots \ n-1
\\
\cdots  \ \cdots \ \cdots
\\
1 \ \ 2 \ \	 \cdots \ n-1
\\
1 \ \ 2 \ \cdots
\end{array}
\right| +
\\
+
\sum_{i = \mu_n+1}^{\mu_{n-1}}
\
(-1)^{(n-1) + (\mu_n - 1)n + (i-\mu_n-1)(n-1) +(n-2)}  
\left(
\begin{array}{lllll}
1 \ \ 2 \ \ \cdots \ n-1 \ \ \ \alpha_p
\\
1 \ \ 2  \ \ \cdots \ n-1 \ \ \ n 
\\
\cdots  \ \cdots \ \cdots
\\
1 \ \ 2  \ \ \cdots \  \ \cdots  \  \  \  \ \ n
\\
1 \ \ 2 \ \  \cdots \ n-1
\\
\cdots  \ \cdots \ \cdots
\\
1 \ \ 2 \ \  \cdots \ \ \ \alpha_p
\\
\cdots  \ \cdots \ \cdots
\\
1 \ \ 2 \ \	 \cdots \ n-1
\\
1 \ \ 2 \ \cdots
\end{array}
\right|,
\end{multline*}
where the tableaux in the two sums are the tableaux with the second occurence of
$\alpha_p$ in the $i$th row.

By the \emph{Straightening Law} of Grosshans, Rota and Stein ( \cite{rota-BR}, Proposition $10$,
see also \cite{Bri-BR}, Thm. $8.1$), each summand in the two sums equals
$$
(-1)^{n-2}
\frac{1}{2}
\left(
\begin{array}{lllll}
1 \ \ 2 \ \ \cdots \ \ \ \alpha_p \ \ \ \ \ \alpha_p
\\
1 \ \ 2  \ \ \cdots \  \   n-1  \  \ n
\\
\cdots  \ \cdots \ \cdots
\\
1 \ \ 2  \ \ \cdots \  \ \cdots  \  \ \ \ \ n
\\
1 \ 2 \ \cdots \ 
\\
\cdots  \ \cdots \ \cdots
\end{array}
\right|
$$
and, hence,
$$
e_{\alpha_p n-1} \cdot \left(
\begin{array}{lllll}
1 \ \ 2 \ \ \cdots \ n-1 \ \alpha_p
\\
1 \ \ 2  \ \ \cdots \  \ \cdots  \  \  \ n
\\
\cdots  \ \cdots \ \cdots
\\
\cdots  \cdots \cdots
\\
1 \ 2 \ \cdots \ \cdots
\\
\cdots  \ \cdots \ \cdots
\end{array}
\right| 
=
(-1)^{n-2} \frac{(\mu_{n-1} + 1)}{2}
\left(
\begin{array}{lllll}
1 \ \ 2 \ \ \cdots \ \ \ \alpha_p \ \ \ \ \ \alpha_p
\\
1 \ \ 2  \ \ \cdots \  \   n-1  \  \ n
\\
\cdots  \ \cdots \ \cdots
\\
1 \ \ 2  \ \ \cdots \  \ \cdots  \  \  \ n
\\
1 \ 2 \ \cdots \ 
\\
\cdots  \ \cdots \ \cdots
\end{array}
\right|.
$$

By iterating this argument, we obtain:
\begin{multline*}
e_{\alpha_p  j} \cdot
\big(\frac{1}{(n-j)!} 
\left(
\begin{array}{lllll}
1 \ \ 2 \ \ \cdots \ j \ \ \ \ \ \ \alpha_p^{n-j}
\\
1 \ \ 2 \ \ \cdots \ j \ \ \cdots    \ n
\\
\cdots  \ \cdots \ \cdots
\\
1 \ \ 2 \ \ \cdots \  j \ \ \cdots    \ n
\\
1 \ \ 2 \ \cdots \ 
\\
\cdots  \ \cdots \ 
\end{array}
\right|\big) =
\\
= (-1)^{j-1} \ 
\frac{\mu_j + n -j}{(n-j+1)!} \
\left(
\begin{array}{lllll}
1 \ \ 2 \ \ \cdots \ j-1 \ \ \ \ \ \ \ \ \ \alpha_p^{n-j+1}
\\
1 \ \ 2 \ \ \cdots \ j-1 \ \ \ j  \ \cdots  \  \ n
\\
\cdots  \ \cdots \ \cdots
\\
1 \ \ 2 \ \ \cdots \ j-1 \ \ \  j \  \cdots  \  \ n
\\
1 \ \  2 \ \cdots \ 
\\
\cdots  \ \cdots \ 
\end{array}
\right|.
\end{multline*}

By iterating this procedure,
\begin{multline*}
e_{\alpha_p 1} \cdots e_{\alpha_p n} \cdot (Der_{\tilde{\mu}} | Der^P_{\tilde{\mu}})
=
\\
= \frac{(-1)^{n \choose 2}} {n!} \ (\mu_1+n-1)(\mu_2+n-2) \cdots \mu_n
\left(
\begin{array}{llll}
\alpha_p \ \ \alpha_p \ \ \cdots \ \ \alpha_p
\\
1        \ \ \  2    \  \ \ \ \cdots \  \ n
\\
\cdots  \ \cdots \ \cdots
\\
1        \ \ \  2    \  \ \ \ \cdots \  \ n
\\
1 \ \ \ 2 \ \cdots \ 
\\
\cdots  \ \cdots \ 
\end{array}
\right|
\end{multline*}
and
\begin{multline*}
 e_{Coder_{n^p}, Der_{n^p}} \cdot (Der_{\tilde{\mu}} | Der^P_{\tilde{\mu}})
=
\\
=
\left( \prod_{i = 0}^{p -1} \ (\mu_1 - i + n -1) \cdots
(\mu_n - i) \right)    \
\frac{(-1)^{{n \choose 2} p}} {(n!)^p} \ 
\left(
\begin{array}{llll}
\alpha_p \ \ \ \ \ \alpha_p\ \ \ \ \cdots \ \ \ \alpha_p
\\
\alpha_{p-1} \ \ \alpha_{p-1}\ \ \cdots \ \ \alpha_{p-1}
\\
\cdots  \ \ \ \cdots \ \ \ \ \ \cdots
\\
\alpha_1 \ \ \ \  \alpha_1 \ \ \ \ \ \ \cdots \ \   \alpha_1
\\
1 \ \ \ \ \ 2 \ \ \ \cdots 
\\
\cdots  \ \ \cdots \ 
\end{array}
\right|
=
\\
=
\left( \prod_{i = 0}^{p -1} \ (\mu_1 - i + n -1) \cdots
(\mu_n - i)  \right)  \
\frac{(-1)^{{n \choose 2} p+{p \choose 2} n}} {(n!)^p} \ 
\left(
\begin{array}{llll}
\alpha_1 \ \ \ \ \ \alpha_1 \  \ \ \ \cdots \ \ \ \alpha_1
\\
\cdots  \ \ \  \cdots \ \ \ \ \cdots
\\
\alpha_{p-1} \ \ \alpha_{p-1}\ \ \cdots \ \ \alpha_{p-1}
\\
\alpha_p \ \ \ \ \ \alpha_p \  \ \ \ \cdots \ \ \ \alpha_p
\\
1 \ \ \ \ \ \  2 \ \ \ \ \ \cdots \
\\
\cdots  \
\end{array}
\right|.
\end{multline*}
Since
\begin{multline*}
e_{Der^*_{n^p}, Coder_{n^p}} \cdot
\left(
\begin{array}{llll}
\alpha_1 \ \ \ \ \ \alpha_1 \  \ \ \ \cdots \ \ \ \alpha_1
\\
\cdots  \ \ \  \cdots \ \ \ \ \cdots
\\
\alpha_{p-1} \ \ \alpha_{p-1}\ \ \cdots \ \ \alpha_{p-1}
\\
\alpha_p \ \ \ \ \ \alpha_p \  \ \ \ \cdots \ \ \ \alpha_p
\\
1 \ \ \ \ \ \  2 \ \ \ \ \ \cdots \
\\
\cdots  \
\end{array}
\right| \ = \ (-1)^{{n \choose 2} p} (n!)^p \ (Der_{\tilde{\mu}} | Der^P_{\tilde{\mu}})=
\\
=
\mathbf{K_n^p}(v_{\mu}) = \mathbf{K_n^p}\cdot (Der_{\tilde{\mu}} | Der^P_{\tilde{\mu}} ) = 
e_{Der^*_{n^p}, Coder_{n^p}} \ e_{Coder_{n^p}, Der_{n^p}} \cdot (Der_{\tilde{\mu}} | Der^P_{\tilde{\mu}} ) = 
\\
= \left( \prod_{i = 0}^{p -1} \ (\mu_1 - i + n -1) \cdots
(\mu_n - i)  \right)  \
\frac{(-1)^{{n \choose 2} p}} {(n!)^p} \ (-1)^{{p \choose 2} n} \times
\\
\times
e_{Der^*_{n^p}, Coder_{n^p}} \cdot 
\big(
\left(
\begin{array}{llll}
\alpha_1 \ \ \ \ \ \alpha_1 \  \ \ \ \cdots \ \ \ \alpha_1
\\
\cdots  \ \ \  \cdots \ \ \ \ \cdots
\\
\alpha_{p-1} \ \ \alpha_{p-1}\ \ \cdots \ \ \alpha_{p-1}
\\
\alpha_p \ \ \ \ \ \alpha_p \  \ \ \ \cdots \ \ \ \alpha_p
\\
1 \ \ \ \ \ \  2 \ \ \ \ \ \cdots \
\\
\cdots  \
\end{array}
\right|\big) =
\\
= \left( \prod_{i = 0}^{p -1} \ (\mu_1 - i + n -1) \cdots
(\mu_n - i)  \right) (-1)^{{p \choose 2} n}  (Der_{\tilde{\mu}} | Der^P_{\tilde{\mu}} ).
\end{multline*}
Notice that, if $\mu_n < p$, then $\mathbf{K_n^p}(v_{\mu}) =0$.

\end{proof}

\section{Appendix. A glimpse on the superalgebraic method of virtual variables }\label{sec Appendix}

In this section, we summarize the main features of the superalgebraic method of virtual variables.
We follow \cite{Brini5-BR} and \cite{BriniTeolisKosz-BR}.

\subsection{The general linear Lie super algebra $gl(m|n)$}
 Given a vector space $V$ of dimension $n$, we will regard it as a subspace of a $
\mathbb{Z}_2-$graded vector space
 $V_0 \oplus V_1$, where
$
V_1 = V.
$
The  vector spaces
$V_0$  (we assume that 
$dim(V_0)=m$ is ``sufficiently large'') is called
the  {\textit{positive virtual (auxiliary)
vector space}} and $V$ 
is called the {\textit{(negative) proper vector space}}.

 The inclusion $V \subset V_0 \oplus V_1$ induces a natural embedding of the ordinary general 
linear Lie algebra $gl(n)$ of $V_n$ into the
 {\textit{auxiliary}}
general linear Lie {\it{superalgebra}} $gl(m|n)$ of $V_0 \oplus V_1$ (see, e.g. \cite{KAC1-BR}, 
\cite{Scheu-BR}).

Let
$
\mathcal{A}_0 = \{ \alpha_1, \ldots, \alpha_{m_0} \},$  
$\mathcal{L} = \{ x_1, x_2,  \ldots, x_n \}$
denote \emph{fixed  bases} of $V_0$ and $V = V_1$, respectively; 
therefore $|\alpha_s| = 0 \in \mathbb{Z}_2,$
and $ |i|   = 1 \in \mathbb{Z}_2.$

Let
$$
\{ e_{a, b}; a, b \in \mathcal{A}_0 \cup \mathcal{L}  \}, \qquad |e_{a, b}| =
|a|+|b| \in \mathbb{Z}_2
$$
be the standard $\mathbb{Z}_2-$homogeneous basis of the Lie superalgebra $gl(m|n)$ provided by the
elementary matrices. The elements $e_{a, b} \in gl(m|n)$ are $\mathbb{Z}_2-$homogeneous of
$\mathbb{Z}_2-$degree $|e_{a, b}| = |a| + |b|.$

The superbracket of the Lie superalgebra $gl(m_0|m_1+n)$ has the following explicit form:
$$
\left[ e_{a, b}, e_{c, d} \right] = \delta_{bc} \ e_{a, d} - (-1)^{(|a|+|b|)(|c|+|d|)} \delta_{ad}  \ e_{c, b},
$$
$a, b, c, d \in \mathcal{A}_0  \cup \mathcal{L} .$

For the sake of readability, we will frequently write $\mathcal{L} = \{ 1, 2,  \ldots, n \}$ in place of
$\mathcal{L} = \{ x_1, x_2,  \ldots, x_n \}$. 

The elements of the sets $\mathcal{A}_0,  \mathcal{L} $ are called
\emph{positive virtual symbols} and \emph{negative proper symbols},
respectively.

\subsection{The  supersymmetric algebra ${\mathbb C}[M_{m|n,d}]$}\label{ citazione 1}

For the sake of readability, given  $n, d \in \mathbb{Z}^+$, $n \leq d$,
we write
$$
M_{n,d} = \left[ (i|j) \right]_{i=1,\ldots,n, j=1, \ldots,d}=
 \left(
 \begin{array}{ccc}
 (1|1) & \ldots & (1|d) \\
 \vdots  &        & \vdots \\
 (n|1) & \ldots & (n|d) \\
 \end{array}
 \right)
$$
in place of
$$
M_{n,d} = \left[ x_{ij} \right]_{i=1,\ldots,n; j=1,\ldots,d}=
 \left[
 \begin{array}{ccc}
 x_{11} & \ldots & x_{1d} \\
 x_{21} & \ldots & x_{2d} \\
 \vdots  &        & \vdots \\
 x_{n1} & \ldots & x_{nd} \\
 \end{array}
 \right].
$$ (compare with eq. (\ref{matrix})) and, consistently,
$$
{\mathbb C}[M_{n,d}] =    {\mathbb C}[(i|j)]_{i=1,\ldots,n, j=1,\ldots,d}
$$
in place of
$$
{\mathbb C}[M_{n,d}] =    {\mathbb C}[x_{ij}]_{i=1,\ldots,n, j=1,\ldots,d}
$$
for the polynomial algebra in the (commutative)
entries $(i|j)$ of the matrix $M_{n,d}$.

We regard the commutative algebra ${\mathbb C}[M_{n,d}]$
as a subalgebra of the \textit{``auxiliary'' supersymmetric algebra}
$$
{\mathbb C}[M_{m|n,d}] 
$$
generated by the ($\mathbb{Z}_2$-graded) variables 
$$
(a|j), \quad a \in  \mathcal{A}_0 \cup \mathcal{L} , 
\quad j \in \mathcal{P}  = \{j=1, \ldots,d; |j|=1 \in \mathbb{Z}_2 \},
$$
with $|(a|j)| =  |a|+|j| \in \mathbb{Z}_2 $,
subject to the commutation relations:
$$
(a|h)(b|k) = (-1)^{|(a|h)||(b|k)|} \ (b|k)(a|h).
$$
In plain words, ${\mathbb C}[M_{m|n,d}]$ is the free supersymmetric algebra 
$$
{\mathbb C}\big[ (\alpha_s|j),  (i|j) \big]
$$
 generated by the ($\mathbb{Z}_2$-graded) variables $(\alpha_s|j),  (i|j)$,
$j = 1, 2, \ldots, d$,
where all the variables commute each other, with the exception of
pairs of variables $(\alpha_s|j), (\alpha_t|j)$ that skew-commute:
$$
(\alpha_s|j) (\alpha_t|j) = - (\alpha_t|j) (\alpha_s|j).
$$

In the standard notation of multilinear algebra, we have:
\begin{align*}
{\mathbb C}[M_{m|n,d}]
& \cong \Lambda \big[ V_0 \otimes P_d \big]
\otimes      {\mathrm{Sym}} \big[ V_1  \otimes P_d \big] 
\end{align*}
where $P_d = (P_d)_1$ denotes the trivially  $\mathbb{Z}_2-$graded  vector space with distinguished basis 
$\mathcal{P}  = \{j=1, \ldots,d; |j|=1 \in \mathbb{Z}_2 \}.$

\subsection{Left superderivations and left superpolarizations}\label{citazione 2}

A {\it{left superderivation}} $D^{\textit{l}}$ ($\mathbb{Z}_2-$homogeneous of degree $|D^{\textit{l}}|$) (see, e.g. \cite{Scheu-BR}, \cite{KAC1-BR}) on
${\mathbb C}[M_{m|n,d}]$ is an element of the superalgebra $End_\mathbb{C}[\mathbb{C}[M_{m|n,d}]]$
that satisfies "Leibniz rule"
$$
D^{\textit{l}}(\textbf{p} \cdot \textbf{q}) = D^{\textit{l}}(\textbf{p}) \cdot \textbf{q} + 
(-1)^{|D^{\textit{l}}||\textbf{p}|} \textbf{p} \cdot D^{\textit{l}}(\textbf{q}),
$$
for every $\mathbb{Z}_2-$homogeneous of degree $|\textbf{p}|$ element $\textbf{p} \in \mathbb{C}[M_{m|n,d}].$

Given two symbols $a, b \in \mathcal{A}_0  \cup \mathcal{L} $, the {\textit{left superpolarization}} $D^{\textit{l}}_{a,b}$ 
of $b$ to $a$
is the unique {\it{left}} superderivation of ${\mathbb C}[M_{m|n,d}]$ of $\mathbb{Z}_2-$degree 
$|D^{\textit{l}}_{a,b}| = |a| + |b| \in \mathbb{Z}_2$ such that
$$
D^{\textit{l}}_{a,b} \left( (c|j) \right) = \delta_{bc} \ (a|j), \ c \in \mathcal{A}_0                                                \cup \mathcal{L} , \ j = 1, \ldots, n.
$$

Informally, we say that the operator $D^{\textit{l}}_{a,b}$ {\it{annihilates}} the symbol $b$ 
and {\it{creates}} the symbol $a$.

\subsection{The superalgebra ${\mathbb C}[M_{m|n,d}]$ as a $\mathbf{U}(gl(m|n))$-module}\label{citazione 3}

Since
$$
D^{\textit{l}}_{a,b}D^{\textit{l}}_{c,d} -(-1)^{(|a|+|b|)(|c|+|d|)}D^{\textit{l}}_{c,d}D^{\textit{l}}_{a,b} =
\delta_{b,c}D^{\textit{l}}_{a,d} -(-1)^{(|a|+|b|)(|c|+|d|)}\delta_{a,d}D^{\textit{l}}_{c,b},
$$
the map
$$
e_{a,b} \mapsto D^{\textit{l}}_{a,b}, \qquad a, b \in \mathcal{A}_0                                               \cup \mathcal{L} 
$$
is a Lie superalgebra morphism from $gl(m|n)$ to $End_\mathbb{C}\big[\mathbb{C}[M_{m|n,d}]\big]$
and, hence, it uniquely defines a
representation:
$$
\varrho : \mathbf{U}(gl(m|n)) \rightarrow End_\mathbb{C}[\mathbb{C}[M_{m|n,d}]],
$$
where $\mathbf{U}(gl(m|n))$ is the enveloping superalgebra of $gl(m|n)$.

In the following, we always regard the superalgebra $\mathbb{C}[M_{m|n,d}]$ as a $\mathbf{U}(gl(m|n))-$supermodule,
with respect to the action induced by the representation $\varrho$:
$$
e_{a,b} \cdot \mathbf{p} = D^{\textit{l}}_{a,b}(\mathbf{p}),
$$
for every $\mathbf{p} \in {\mathbb C}[M_{m|n,d}].$

We recall that  $\mathbf{U}(gl(m|n))-$module  $\mathbb{C}[M_{m|n,d}]$
is  a semisimple module, whose simple submodules are - up to isomorphism - {\it{Schur supermodules}} 
(see, e.g. \cite{Brini1-BR}, \cite{Brini2-BR}, \cite{Bri-BR}. For a more traditional presentation, see also 
\cite{CW-BR}).

Clearly, $\mathbf{U}(gl(0|n)) = \mathbf{U}(gl(n))$ is a subalgebra of $\mathbf{U}(gl(m|n))$
and the subalgebra $\mathbb{C}[M_{n,d}]$ is a $\mathbf{U}(gl(n))-$submodule of  $\mathbb{C}[M_{m|n,d}]$.

\subsection{The virtual algebra $Virt(m,n)$ and the virtual
presentations of elements in $\mathbf{U}(gl(n))$}\label{citazione 4}

We say that a product
$$
e_{a_m,b_m} \cdots e_{a_1,b_1} \in \mathbf{U}(gl(m|n)), 
\quad a_i, b_i \in \mathcal{A}_0  \cup \mathcal{L} , \ i= 1, \ldots, m
$$
is an {\textit{irregular expression}} whenever
  there exists a right subword
$$e_{a_i,b_i} \cdots e_{a_2,b_2} e_{a_1,b_1},$$
$i \leq m$ and a
virtual symbol $\gamma \in \mathcal{A}_0 $ such that
\begin{equation*}\label{irrexpr-BR}
 \# \{j;  b_j = \gamma, j \leq i \}  >  \# \{j;  a_j = \gamma, j < i \}.
\end{equation*}

The meaning of an irregular expression in terms of the action of  $\mathbf{U}(gl(m|n))$  
by left superpolarization on
the algebra $\mathbb{C}[M_{m|n,d}]$ is that there exists a
virtual symbol $\gamma$ and a right subsequence in which the symbol $\gamma$ is \emph{annihilated} 
more times than it was already \emph{created} and, therefore, the action of an irregular expression
on the algebra $\mathbb{C}[M_{n,d}]$ is \emph{zero}. 

\begin{example}
Let $\gamma \in  \mathcal{A}_0 $ and $x_i, x_j \in \mathcal{L}.$ The product
$$
e_{\gamma,x_j} e_{x_i,\gamma} e_{x_j,\gamma} e_{\gamma,x_i}
$$
is an irregular expression.
\end{example}\qed

Let $\mathbf{Irr}$   be
the {\textit{left ideal}} of $\mathbf{U}(gl(m|n))$ generated by the set of
irregular expressions.

\begin{proposition}
The superpolarization action
of any element of $\mathbf{Irr}$ on the subalgebra $\mathbb C[M_{n,d}] \subset \mathbb{C}[M_{m|n,d}]$ - via the representation $\varrho$ -
is identically zero.
\end{proposition}

\begin{proposition}
The sum ${\mathbf{U}}(gl(0|n)) + \mathbf{Irr}$ is a direct sum of vector subspaces of $\mathbf{U}(gl(m|n)).$
\end{proposition}

\begin{proposition}
The direct sum vector subspace $\mathbf{U}(gl(0|n)) \oplus \mathbf{Irr}$ is a \emph{subalgebra} 
of $\mathbf{U}(gl(m|n)).$
\end{proposition}

The subalgebra
$$
Virt(m,n) = \mathbf{U}(gl(0|n)) \oplus \mathbf{Irr} \subset {\mathbf{U}}(gl(m|n)).
$$
is called the {\textit{virtual algebra}}.

\begin{proposition}
The left ideal  $\mathbf{Irr}$ of ${\mathbf{U}}(gl(m|n))$
is a two sided ideal of $Virt(m,n).$
\end{proposition}

The {\textit{Capelli devirtualization epimorphism}} is the surjection
$$
\mathfrak{p} : Virt(m,n) = \mathbf{U}(gl(0|n)) \oplus \mathbf{Irr} \twoheadrightarrow 
\mathbf{U}(gl(0|n)) = \mathbf{U}(gl(n))
$$
with $Ker(\mathfrak{p}) = \mathbf{Irr}.$

Any element in $\textbf{M} \in Virt(m,n)$ defines an element in
$\textbf{m} \in \mathbf{U}(gl(n))$ - via the map $\mathfrak{p}$ -
 and $\textbf{M}$ is called a \textit{virtual
presentation} of $\textbf{m}$.

Furthermore,
\begin{proposition}
The subalgebra $\mathbb C[M_{n,d}] \subset \mathbb{C}[M_{m|n,d}]$ is  invariant with respect to the action
of the subalgebra
$
Virt(m, n).
$
\end{proposition}
\begin{proposition}\label{virtual action}
For every element $\mathbf{m} \in {\mathbf{U}}(gl(n))$, the action of $\mathbf{m}$ on 
the subalgebra $\mathbb C[M_{n,d}]$
is the same of the action of any of its virtual presentation $\mathbf{M}  \in Virt(m,n).$
In symbols, 
$$
if  \quad
\mathfrak{p}(\mathbf{M}) = \mathbf{m}
\quad
then
\quad
\mathbf{m} \cdot \mathbf{P} = \mathbf{M} \cdot \mathbf{P},
\quad
for \ every \ \mathbf{P} \in \mathbb C[M_{n,d}].
$$
\end{proposition}
Since the map $\mathfrak{p}$  a surjection, any element
$\mathbf{m} \in \mathbf{U}(gl(n))$ admits several virtual
presentations. In the sequel, we even take virtual presentations
as the \emph{definition} of special elements in $\mathbf{U}(gl(n)),$ 
and this method will turn out to be quite effective.

The superalgebra ${\mathbf{U}}(gl(m|n))$ is a Lie module with respect 
to the adjoint representation $Ad_{gl(m|n)}$. Since $gl(n) = gl(0|n)$ 
is a Lie subalgebra of $gl(m|n)$, then ${\mathbf{U}}(gl(m|n))$ is a $gl(n)-$module 
with respect to the adjoint action $Ad_{gl(n)}$ of $gl(n)$.

\begin{proposition} The virtual algebra $Virt(m,n)$ is a submodule 
of ${\mathbf{U}}(gl(m|n))$ with respect to the adjoint action $Ad_{gl(n)}$ of $gl(n)$.
\end{proposition}

\begin{proposition}\label{rappresentazione aggiunta-BR} The Capelli  epimorphism 
$$
\mathfrak{p} : Virt(m,n)  \twoheadrightarrow \mathbf{U}(gl(n))
$$ is an $Ad_{gl(n)}-$\emph{equivariant} map.
\end{proposition}

\begin{corollary}\label{centrality gen} The isomorphism $\mathfrak{p}$ maps
any  $Ad_{gl(n)}-$invariant element $\mathbf{m} \in Virt(m,n)$  to a \emph{central} 
element of $\mathbf{U}(gl(n))$.
\end{corollary}

\textit{Balanced monomials} are  elements of the algebra ${\mathbf{U}}(gl(m|n))$
 of the form:
\begin{itemize}\label{defbalanced monomials-BR}
\item [--] $e_{{i_1},\gamma_{p_1}} \cdots e_{{i_k},\gamma_{p_k}} \cdot
e_{\gamma_{p_1},{j_1}} \cdots e_{\gamma_{p_k},{j_k}},$
\item [--]
$e_{{i_1},\theta_{q_1}} \cdots e_{{i_k},\theta_{q_k}} \cdot
e_{\theta_{q_1},\gamma_{p_1}} \cdots e_{\theta_{q_k},\gamma_{p_k}} \cdot
e_{\gamma_{p_1},{j_1}} \cdots e_{\gamma_{p_k},{j_k}},$
\item [--] and so on,
\end{itemize}
where
$i_1, \ldots, i_k, j_1, \ldots, j_k \in L,$
i.e., the $i_1, \ldots, i_k, j_1, \ldots, j_k$ are $k$
proper (negative) symbols, and the
$\gamma_{p_1}, \ldots, \gamma_{p_k}, \ldots, \theta_{q_1}, \ldots, \theta_{q_k}, \ldots$ are
virtual symbols.
In plain words, a balanced monomial is product of two or more factors  where the
rightmost one  \textit{annihilates} (by superpolarization)
the $k$ proper symbols $ j_1, \ldots, j_k$ and
\textit{creates} (by superpolarization) some virtual symbols;
 the leftmost one  \textit{annihilates} all the virtual symbols
and \textit{creates} the $k$ proper symbols $ i_1, \ldots, i_k$;
between these two factors, there might be further factors that annihilate
 and create  virtual symbols only.

\begin{proposition}
Every balanced monomial belongs to $Virt(m,n)$. Hence,
the Capelli epimorphism $\mathfrak{p}$ maps  balanced monomials to elements of $\mathbf{U}(gl(n)).$
\end{proposition}

\subsection{Bitableaux monomials and Capelli bitableaux in ${\mathbf{U}}(gl(n))$}\label{citazione 5}

We will introduce two classes of remarkable elements of the enveloping algebra ${\mathbf{U}}(gl(n))$, that
we call {\textit{bitableaux monomials}}, {\textit{Capelli bitableaux}}, respectively.

 Let $\lambda \vdash h$ be a partition, and label the boxes of its Ferrers diagram
with the numbers $1, 2, \ldots , h$ in the following way:
$$
\begin{array}{lllll}
1 & 2 &  \cdots & \cdots & \lambda_1 \\
\lambda_1 + 1 &  \lambda_1 + 2 & \cdots & \lambda_1 + \lambda_2 &    \\
\cdots & \cdots & \cdots &  &  \\
\cdots & \cdots & h &  &  \\
\end{array}.
$$
A {\textit{Young tableau}} $T$ of shape $\lambda$ over the alphabet 
$\mathcal{A} =  \mathcal{A}_0 \cup \mathcal{L} $ 
is a map $T : \underline{h} = \{1, 2, \ldots , h \}  \rightarrow \mathcal{A}$; the element $T(i)$
is the symbol in the cell $i$ of the tableau $T$.

The sequences
$$
\begin{array}{l}
T(1)  T(2) \cdots  T(\lambda_1),
\\
T(\lambda_1 + 1)  T(\lambda_1 + 2)  \cdots  T(\lambda_1 + \lambda_2),
\\
 \ldots \ldots
\end{array}
$$
are called the {\textit{row words}} of the Young tableau $T$.

We will also denote a Young tableau
by its sequence of rows words, that is $T = (\omega_1, \omega_2, \ldots, \omega_p)$.
Furthermore, the {\textit{word of the tableau}} $T$ is the concatenation
\begin{equation*}\label{word}
w(T) = \omega_1\omega_2 \cdots \omega_p.
\end{equation*}

The {\textit{content}} of a tableau $T$ is the function $c_T : \mathcal{A} \rightarrow \mathbb{N}$,
$$
c_T(a) = \sharp \{i \in \underline{h}; \ T(i) = a \}.
$$

Given a shape/partition $\lambda$, 
we assume that $|\mathcal{A}_0| = m \geq \widetilde{\lambda}_1$, 
where  $\widetilde{\lambda}$ denotes the conjugate shape/partition of $\lambda$.
Let us denote by $\alpha_1, \ldots, \alpha_p \in \mathcal{A}_0$  an \emph{arbitrary} family 
of \emph{distinct positive symbols}.
Set
\begin{equation}\label{Deruyts and Coderuyts}
C_{\lambda}^* = \left(
\begin{array}{llllllllllllll}
\alpha_1 \ldots    \ldots     \ldots     \alpha_1                           \\
\alpha_2   \ldots  \ldots               \alpha_2 \\
 \ldots  \ldots   \\
\alpha_p \ldots \alpha_p
\end{array} \right).
\end{equation}

The tableaux of  kind (\ref{Deruyts and Coderuyts}) are called  {\it{virtual  Coderuyts   tableaux}}
of shape $\lambda,$.

Let $S$ and $T$ be two Young tableaux of same shape $\lambda \vdash h$ on
the  alphabet $ \mathcal{A}_0  \cup \mathcal{L} $:

\begin{equation*}\label{bitableaux}
S = \left(
\begin{array}{llllllllllllll}
z_{i_1}  \ldots    \ldots     \ldots     z_{i_{\lambda_1}}     \\
z_{j_1}   \ldots  \ldots               z_{j_{\lambda_2}} \\
 \ldots  \ldots   \\
z_{s_1} \ldots z_{s_{\lambda_p}}
\end{array}
\right), \qquad
T = \left(
\begin{array}{llllllllllllll}
z_{h_1}  \ldots    \ldots     \ldots     z_{h_{\lambda_1}}    \\
z_{k_1}   \ldots  \ldots               z_{k_{\lambda_2}} \\
 \ldots  \ldots   \\
z_{t_1} \ldots z_{t_{\lambda_p}}
\end{array}
\right).
\end{equation*}

To the pair $(S,T)$, we associate the {\it{bitableau monomial}}:
\begin{equation*}\label{BitMon}
e_{S,T} =
e_{z_{i_1}, z_{h_1}}\cdots e_{z_{i_{\lambda_1}}, z_{h_{\lambda_1}}}
e_{z_{j_1}, z_{k_1}}\cdots e_{z_{j_{\lambda_2}}, z_{k_{\lambda_2}}}
 \cdots  \cdots
e_{z_{s_1}, z_{t_1}}\cdots e_{z_{s_{\lambda_p}}, z_{t_{\lambda_p}}}
\end{equation*}
in ${\mathbf{U}}(gl(m|n)).$

Given a pair of Young tableaux  $S, T$ of the same shape $\lambda$ on the proper alphabet $L$, 
consider the elements
$$
e_{S,C_{\lambda}^*} \ e_{C_{\lambda}^*,T} \in {\mathbf{U}}(gl(m|n)).
$$

Since these elements  are \emph{balanced monomials} in
${\mathbf{U}}(gl(m|n))$, then they belong to the \emph{virtual subalgebra} $Virt(m,n)$.

Hence, we can consider their images in ${\mathbf{U}}(gl(n))$ with respect to the Capelli epimorphism $\mathfrak{p}$.

We set
\begin{equation}\label{determinantal}
 \mathfrak{p} \Big( e_{S,C_{\lambda}^*} \ e_{C_{\lambda}^*,T}    \Big) = [S|T]    \in {\mathbf{U}}(gl(n)),
\end{equation}
and call the element $[S|T]$ a {\textit{Capelli bitableau}}.

The elements defined in (\ref{determinantal}) do not
depend on the choice of the virtual  Coderuyts tableau $C_{\lambda}^*$.

\subsection{Biproducts and bitableaux in ${\mathbb C}[M_{m|n,d}]$}\label{citazione 6}

Embed the algebra
$$
{\mathbb C}[M_{m|n,d}] = {\mathbb C}[(\alpha_s|j), (i|j)]
$$
into the (supersymmetric) algebra ${\mathbb C}[(\alpha_s|j), (i|j), (\gamma|j)]$
generated by the ($\mathbb{Z}_2$-graded) variables $(\alpha_s|j), (i|j), (\gamma|j)$,
$j = 1, 2, \ldots, d$,
 where
 $$
 |(\gamma|j)| = 1 \in \mathbb{Z}_2 \ \  for \ every \ j = 1, 2, \ldots, d,
 $$
and denote by $D^l_{z_i,\gamma}$ the superpolarization of $\gamma$ to  $z_i.$

Let $\omega = z_1z_2 \cdots z_p$ be a word on     $ \mathcal{A}_0 \cup  \mathcal{L} $, 
and $\varpi = j_{t_1}j_{t_2} \cdots j_{t_q}$ a word
on the alphabet
$P = \{1, 2, \ldots, d \}$. The {\it{biproduct}}
$$
(\omega|\varpi) = (z_1z_2 \cdots z_p|j_{t_1}j_{t_2} \cdots j_{t_q})
$$
is the element
$$
D^l_{z_1,\gamma}D_{z_2,\gamma} \cdots D^l_{z_p,\gamma} \Big( (\gamma|j_{t_1})(\gamma|j_{t_2}) \cdots
(\gamma|j_{t_q}) \Big) \in {\mathbb C}[M_{m|n,d}]
$$
if $p = q$ and is set to be zero otherwise.

\begin{claim} 
The biproduct $(\omega|\varpi) = (z_1z_2 \cdots z_p|j_{t_1}j_{t_2} \cdots j_{t_q})$ is supersymmetric in the  
$z$'s and
skew-symmetric in the  $j$'s.
In symbols
\begin{enumerate}

\item
$
(z_1 z_2 \cdots z_i z_{i+1} \cdots z_p|j_{t_1} j_{t_2} \cdots j_{t_q}) =
\\ \null \hfill (-1)^{|z_i| |z_{i+1}|}
(z_1 z_2 \cdots z_{i+1} z_i \cdots z_p|j_{t_1} j_{t_2} \cdots j_{t_q})
$

\item
$
(z_1z_2 \cdots z_iz_{i+1} \cdots z_p|j_{t_1}j_{t_2} \cdots j_{t_i}j_{t_{i+1}} \cdots j_{t_q}) =
\\ \null \hfill
- (z_1z_2 \cdots z_iz_{i+1} \cdots z_p|j_{t_1} \cdots j_{t_{i+1}}j_{t_i} \cdots j_{t_q}).
$
\end{enumerate}

\end{claim}

\begin{proposition}{\textbf{(Laplace expansions)}}\label{Laplace expansions}
We have
\begin{enumerate}

\item
$
(\omega_1\omega_2|\varpi) = \Sigma_{(\varpi)} \ (-1)^{|\varpi_{(1)}| |\omega_2|} \ 
(\omega_1|\varpi_{(1)})(\omega_2|\varpi_{(2)}).
$

\item
$
(\omega|\varpi_1\varpi_2) = \Sigma_{(\omega)} \  (-1)^{|\varpi_1| |\omega_{(2)}|}  \ 
(\omega_{(1)}|\varpi_1)(\omega_{(2)}|\varpi_2.)
$

\end{enumerate}
where
$$
\bigtriangleup(\varpi) = \Sigma_{(\varpi)}  \ \varpi_{(1)} \otimes \varpi_{(2)}, \quad \bigtriangleup(\omega)
= \Sigma_{(\omega)} \ \omega_{(1)} \otimes \omega_{(2)}
$$
denote the coproducts in the \emph{Sweedler notation}  of the elements $\varpi$ and $\omega$ in the
supersymmetric Hopf algebra of $W$
\emph{(see, e.g. \cite{Bri-BR})}   and in
the free exterior Hopf algebra generated by
$j = 1, 2, \ldots, d$, respectively.

\end{proposition}

Let $\omega = i_1i_2 \cdots i_p$, $\varpi = j_1j_1 \cdots j_p$ be  words on the negative alphabet 
$\mathcal{L} = \{1, 2, \ldots, n \}$ and on  the negative alphabet 
$\mathcal{P} = \{1, 2, \ldots, d \}$.

From Proposition \ref{Laplace expansions}, we infer

\begin{corollary}
The {\it{biproduct}} of the two words $\omega$ and $\varpi$
\begin{equation*}\label{biproduct}
(\omega|\varpi) = (i_1i_2 \cdots i_p|j_1j_2 \cdots j_p)
\end{equation*}
is the  \emph{signed minor}:
$$
(\omega|\varpi) = (-1)^{p \choose 2} \ 
det \Big( \ (i_r|j_s) \ \Big)_{r, s = 1, 2, \ldots, p} \in {\mathbb C}[M_{n,d}]. 
$$
\end{corollary}

Following the notation introduced in the previous sections, let
$$
Super[V_0 \oplus V_1] = Sym[V_0] \otimes \Lambda[V_1]
$$
denote the {\it{(super)symmetric}} algebra of the space
$$
V_0 \oplus V_1
$$ (see, e.g. \cite{Scheu-BR}).

By multilinearity, the algebra
$Super[V_0 \oplus V_1]$ is the same as the superalgebra  $Super[\mathcal{A}_0 \cup \mathcal{L} ]$ 
generated by the "variables"
$$
 \alpha_1, \ldots, \alpha_{m_0}  \in  \mathcal{A}_0, \quad 1, \ldots, n \in L,
$$
modulo the congruences
$$
z z' = (-1)^{|z| |z'|} z' z, \quad z, z' \in \mathcal{A}_0   \cup \mathcal{L} .
$$
Let $d^l_{z, z'}$ denote the (left)polarization operator of $z'$ to $z$ on
$$
Super[W] = Super[\mathcal{A}_0  \cup \mathcal{L} ],
$$
that is the unique superderivation of $\mathbb{Z}_2$-degree
$$
|z| + |z'| \in \mathbb{Z}_2
$$
such that
$$
d^l_{z, z'} (z'') = \delta_{z', z''} \cdot z,
$$
for every $z, z', z'' \in \mathcal{A}_0  \cup \mathcal{L} .$

Clearly, the map
$$
e_{z, z'}  \rightarrow   d^l_{z, z'}
$$
is a Lie superalgebra map and, therefore, induces a structure of
$$gl(m|n)-module$$
on
$Super[\mathcal{A}_0  \cup \mathcal{L} ] = Super[V_0 \oplus V_1 ].$

\begin{proposition}\label{polarization biproduct}

Let $\varpi = j_{t_1}j_{t_2} \cdots j_{t_q}$ be a word on $P = \{1, 2, \ldots, d \}$.
The map
$$
\Phi_\varpi : \omega \mapsto (\omega|\varpi),
$$
$\omega$ any word on $\mathcal{A}_0 \cup \mathcal{L} $, uniquely defines $gl(m|n)-$equivariant linear operator
$$
\Phi_\varpi : Super[\mathcal{A}_0  \cup \mathcal{L} ] \rightarrow {\mathbb C}[M_{m|n,d}],
$$
that is
\begin{equation*}\label{polarization action}
\Phi_\varpi \big( e_{z, z'} \cdot \omega \big) =\Phi_\varpi \big( d^l_{z, z'}(\omega) \big) =
 D^l_{z, z'} \big( (\omega|\varpi) \big) =
 e_{z, z'} \cdot (\omega|\varpi),
\end{equation*}
for every $z, z' \in \mathcal{A}_0  \cup \mathcal{L} .$
\end{proposition}

With a slight abuse of notation, we will write (\ref{polarization action}) in the form
\begin{equation}\label{abuse}
D^l_{z, z'} \big( (\omega|\varpi) \big) = ( d^l_{z, z'}(\omega)|\varpi).
\end{equation}

Let $S = (\omega_1, \omega_2, \ldots, \omega_p$ and 
$T = (\varpi_1, \varpi_2, \ldots, \varpi_p)$ be Young tableaux on
$\mathcal{A}_0   \cup \mathcal{L} $ and $P = \{1, 2, \ldots, d \}$ of shapes $\lambda$ and $\mu$, respectively.

If $\lambda = \mu$, the {\it{Young bitableau}} $(S|T)$ is the element of ${\mathbb C}[M_{m|n,d}]$ defined as follows:
$$
(S|T) =
\left(
\begin{array}{c}
\omega_1\\ \omega_2\\ \vdots\\ \omega_p
\end{array}
\right| \left.
\begin{array}{c}
\varpi_1\\ \varpi_2\\ \vdots\\  \varpi_p
\end{array}
\right)
= \pm \ (\omega_1)|\varpi_1)(\omega_2)|\varpi_2) \cdots (\omega_p)|\varpi_p),
$$
where
$$
\pm  = (-1)^{|\omega_2||\varpi_1|+|\omega_3|(|\varpi_1|+|\varpi_2|)+ \cdots +|\omega_p|(|\varpi_1|+|\varpi_2|+\cdots+|\varpi_{p-1}|)}.
$$

If $\lambda \neq \mu$, the {\it{Young bitableau}} $(S|T)$ is set to be zero.

By naturally extending the slight abuse of notation (\ref{abuse}), the action of any polarization on bitableaux
can be explicitly described:

\begin{proposition}\label{action on tableaux} 
Let $z, z' \in \mathcal{A}_0 \cup \mathcal{L} $,  and let
$S = (\omega_1, \ldots, \omega_p) $, $T =
(\varpi_1, \ldots, \varpi_p)$. We have the
following identity:
\begin{align*}
e_{z, z'} \cdot (S\,|\,T) \ & = 
\ D^l_{z, z'} \ \big( \left(
\begin{array}{c}
\omega_1\\ \omega_2\\ \vdots\\ \omega_p
\end{array}
\right| \left.
\begin{array}{c}
\varpi_1\\ \varpi_2\\ \vdots\\  \varpi_p
\end{array}
\right) \big) \\ &
= \ \sum_{s=1}^p  \
(-1)^{(|z| + |z'|)\epsilon_s}
\ \left(
\begin{array}{c}
\omega_1\\ \omega_2\\ \vdots\\  d^l_{z,
z'}(\omega_s)\\ \vdots \\ \omega_p
\end{array}
\right| \left.
\begin{array}{c}
\varpi_1\\ \varpi_2\\ \vdots\\
\vdots \\ \vdots\\ \varpi_p
\end{array}
\right),
\end{align*}
where
$$
\epsilon_1 = 1, \quad    \epsilon_s = |\omega_1| + \cdots + |\omega_{s-1}|, \quad s = 2,
\ldots, p.
$$
\end{proposition}

\begin{example} Let $\alpha_i \in \mathcal{A}_0$, $1,  2, 3, 4 \in L$, $|D_{\alpha_i, 2}| = 1$. Then
$$
e_{\alpha_i, 2} \cdot
\left(
\begin{array}{lll}
1 \ 3 \ 2\\
 2 \ 3 \\
4 \ 2
\end{array}
\right| \left.
\begin{array}{lll}
1 \ 2 \ 3 \\
2 \ 3 \\
3 \ 1
\end{array}
\right) =
D^l_{\alpha_i, 2} \ \big(
\left(
\begin{array}{lll}
1 \ 3 \ 2\\
 2 \ 3 \\
4 \ 2
\end{array}
\right| \left.
\begin{array}{lll}
1 \ 2 \ 3 \\
2 \ 3 \\
3 \ 1
\end{array}
\right) \big) =
$$
$$ =
\left(
\begin{array}{lll}
1 \ 3 \ \alpha_i\\
 2 \ 3\\
4 \ 2
\end{array}
\right| \left.
\begin{array}{lll}
1 \ 2 \ 3 \\
2 \ 3 \\
3 \ 1
\end{array}
\right) -
\left(
\begin{array}{lll}
1 \ 3 \ 2\\
 \alpha_i \ 3\\
4 \ 2
\end{array}
\right| \left.
\begin{array}{lll}
1 \ 2 \ 3 \\
2 \ 3 \\
3 \ 1
\end{array}
\right) +
\left(
\begin{array}{lll}
1 \ 3 \ 2\\
 2 \ 3\\
4 \ \alpha_i
\end{array}
\right| \left.
\begin{array}{lll}
1 \ 2 \ 3 \\
2 \ 3 \\
3 \ 1
\end{array}
\right).
$$

\end{example}

\end{document}